\documentclass{amsart}%
\usepackage{amssymb}
\usepackage{amsmath}
\usepackage{amsfonts}
\usepackage[all,cmtip]{xy}
\usepackage[margin=1.5in]{geometry}
\usepackage{hyperref}
\usepackage{graphicx}%
\newtheorem{theorem}{Theorem}
\newtheorem{theoremannounce}{Theorem}
\numberwithin{theorem}{section}
\theoremstyle{plain}
\newtheorem*{acknowledgement}{Acknowledgement}

\newtheorem{definition}[theorem]{Definition}

\newtheorem{lemma}[theorem]{Lemma}

\theoremstyle{remark}

\newtheorem{example}[theorem]{Example}

\newtheorem{observe}{Observation}

\newtheorem{remark}[theorem]{Remark}

\numberwithin{equation}{section}

\begin{document}
\title[Braided categorical groups and associators]{Braided categorical groups and\linebreak strictifying associators}
\author{Oliver Braunling}
\address{Mathematical Institute, University of Freiburg, Ernst-Zermelo-Strasse 1,
79104, Freiburg im Breisgau, Germany}
\thanks{The author was supported by DFG GK1821 \textquotedblleft Cohomological Methods
in Geometry\textquotedblright.}
\keywords{Braided categorical group, Picard groupoid, strictification, skeletalization, associator}

\begin{abstract}
A key invariant of a braided categorical group is its quadratic form,
introduced by Joyal and Street. We show that the categorical group is braided
equivalent to a simultaneously skeletal and strictly associative one if and
only if the polarization of this quadratic form is the symmetrization of a
bilinear form. This generalizes the result of Johnson--Osorno that all Picard
groupoids can simultaneously be strictified and skeletalized, except that in
the braided case there is a genuine obstruction.

\end{abstract}
\maketitle

A braided monoidal category is always braided monoidal equivalent to a
skeletal one, and always braided monoidal equivalent to a strictly associative
one. However, typically it is impossible to achieve both properties
simultaneously; even in the more restrictive symmetric monoidal case. But
Johnson and Osorno \cite{MR2981952} have shown that this problem,
surprisingly, disappears for Picard groupoids, i.e. symmetric categorical
groups. We generalize this to the braided case, except that here one faces an
obstruction. In detail:\medskip

A \emph{braided categorical group} is a braided monoidal category
$(\mathsf{C},\otimes)$ which is additionally a groupoid and for all objects
$X\in\mathsf{C}$, there exists an object $X^{-1}$ with an arrow%
\[
X\otimes X^{-1}\overset{\sim}{\longrightarrow}\mathbf{1}%
\]
to the unit object. The same concept is also called a `braided $Gr$-category'
or `braided weak $2$-group'. If the braiding is symmetric, it is called a
\emph{Picard groupoid}.

Joyal and Street \cite{MR1250465} give a classification of braided categorical
groups up to braided monoidal equivalences. They attach to any such an abelian
$3$-cocycle in $H_{ab}^{3}(\pi_{0}\mathsf{C},\pi_{1}\mathsf{C})$.
Eilenberg--Mac\thinspace Lane canonically identify the latter group with
quadratic forms%
\begin{equation}
\operatorname*{tr}:H_{ab}^{3}(G,M)\overset{\cong}{\longrightarrow
}\operatorname*{Quad}(G,M)\text{.} \label{lEMMa}%
\end{equation}
This isomorphism holds for arbitrary abelian groups $G,M$. The subtle
differences between symmetric bilinear forms and quadratic forms cause some
intricacies here unless $2$ is invertible, and this issue lies at the heart of
the entire matter. We call a quadratic form \emph{polar} if its polarization,
i.e. the symmetric bilinear form%
\[
b(x,y):=q(x+y)-q(x)-q(y)\text{,}%
\]
is of the shape $t(x,y)+t(y,x)$ for some bilinear form $t:G\times G\rightarrow
M$ (where $t$ is not required to be symmetric or non-degenerate in any way).
Our main result is as follows.

\begin{theoremannounce}
A braided categorical group $(\mathsf{C},\otimes)$ is braided monoidal
equivalent to one which is simultaneously strictly associative and skeletal if
and only if its quadratic form is polar.
\end{theoremannounce}

See Theorem \ref{thm_StrictlyAssocSkeletalization}. This generalizes the
following result and reproves it in a different way.

\begin{theoremannounce}
[Johnson--Osorno \cite{MR2981952}]Every Picard groupoid $(\mathsf{C},\otimes)$
is symmetric monoidal equivalent to one which is simultaneously strictly
associative and skeletal.
\end{theoremannounce}

This follows from our main result since for Picard groupoids the polarization
always vanishes, so one can just take $t(x,y):=0$. We give explicit examples
of non-polar forms. Even in the non-polar situation, one can always add new
objects to the category to make it polar.

\begin{theoremannounce}
For every braided categorical group $(\mathsf{C},\otimes)$, there exists an
essentially surjective and faithful (but typically not full) braided monoidal
functor from another braided categorical group%
\[
(\widehat{\mathsf{C}},\otimes)\longrightarrow(\mathsf{C},\otimes)\text{,}%
\]
surjective on $\pi_{0}$, and an isomorphism on $\pi_{1}$, such that
$(\widehat{\mathsf{C}},\otimes)$ is simultaneously strictly associative and
skeletal. We call $(\widehat{\mathsf{C}},\otimes)$ a \emph{polar cover} of
$(\mathsf{C},\otimes)$.
\end{theoremannounce}

See Theorem \ref{thm_PolarCover}. This construction is not canonical. It has
occasionally been asked\footnote{\textquotedblleft But I had difficulty seeing
in an explicit way how to get an associator and braiding from a quadratic
form.\textquotedblright\ \cite{BAEZPOST}. Quinn \cite[\S 2.5]{MR1734419} or
Galindo--Jaramillo \cite[\S 4.4]{MR3605649} discuss $G$ finite and $M=U(1)$.
Various further results exist in the literature, mostly for $G$ finite and $M$
inside $\mathbb{C}^{\times}$, Wall \cite{MR156890}, Durfee \cite{MR480333}.}
how to find the commutativity constraint explicitly if one is only given the
quadratic form. We answer this in the polar case.

\begin{theoremannounce}
Let $G,M$ be abelian groups. Let $(\beta_{i})_{i\in I}$ be a basis of the
$\mathbb{F}_{2}$-vector space $G/2G$. Let $q\in\operatorname*{Quad}(G,M)$ be a
polar quadratic form and pick a $t:G\times G\rightarrow M$ such that
$b(x,y)=t(x,y)+t(y,x)$. Then%
\[
\overline{q}(x):=q(x)-t(x,x)
\]
defines an element $\overline{q}\in\operatorname*{Hom}_{\mathbb{Z}}(G/2G,M)$
and each $\overline{q}(\beta_{i})$ is $2$-torsion in $M$. Define the abelian
$3$-cocycle $(h,c)\in H_{ab}^{3}(G,M)$ with%
\begin{align*}
h(x,y,z)  &  :=0\\
c(x,y)  &  :=t(x,y)+\sum_{i\in I}\overline{x}_{i}\cdot\overline{y}_{i}%
\cdot\overline{q}(\beta_{i})\text{,}%
\end{align*}
where $\overline{x},\overline{y}$ are the vectors we get under the quotient
map $G\twoheadrightarrow G/2G$, spelled out with respect to the basis
$(\beta_{i})$. This means that $\overline{x}_{i},\overline{y}_{i}\in
\mathbb{F}_{2}$. Then $(h,c)$ maps to $q$ under the isomorphism of Equation
\ref{lEMMa}.
\end{theoremannounce}

Note that while \cite{MR2981952} abstractly proves that Picard groupoids can
be made strictly associative and skeletal by showing that the symmetric
$3$-cohomology class must have a representative with trivial associator, the
above provides an explicit $3$-cocycle having this property.\medskip

When preparing this text, I sometimes got stuck because in a lot of literature
many little (and often admittedly harmless) verifications are left as
\textsl{Exercises}, especially in the older literature, where perhaps most of
this was considered folklore. I found this a little inconvenient, so this text
is written in a quite self-contained way, providing details for many of these
customary omissions.

\begin{acknowledgement}
I thank Niles Johnson and Brad Drew for helpful correspondence.
\end{acknowledgement}

\section{Braided categorical groups}

\begin{definition}
A \emph{braided categorical group} is a braided monoidal category
$(\mathsf{C},\otimes)$ which is additionally a groupoid and for every object
$X\in\mathsf{C}$, we are given an object $X^{-1}$ with an arrow%
\begin{equation}
\varepsilon_{X}:X\otimes X^{-1}\overset{\sim}{\longrightarrow}\mathbf{1}
\label{lefmu7a}%
\end{equation}
to the unit object. Alternative names are: braided $Gr$-categories, braided
weak $2$-groups.
\end{definition}

It may appear more natural to demand the existence of an inversion functor%
\begin{equation}
(-)^{-1}:\mathsf{C}^{op}\longrightarrow\mathsf{C} \label{lefmu7b}%
\end{equation}
sending $X$ to $X^{-1}$ functorially. The definition above does not a priori
guarantee any functoriality in the formation of inverses. However, Laplaza has
shown that in any categorical group the inverses $X^{-1}$ necessarily extend
uniquely to give a functor as in Equation \ref{lefmu7b}, \cite[4.3
Proposition]{MR723395}. This is also discussed in \cite{MR2068521}.

\begin{definition}
If the braiding%
\[
s_{X,Y}:X\otimes Y\overset{\sim}{\longrightarrow}Y\otimes X
\]
is symmetric, i.e. $s_{Y,X}\circ s_{X,Y}=\operatorname*{id}_{X\otimes Y}$ for
all $X,Y\in\mathsf{C}$, then $(\mathsf{C},\otimes)$ is called a \emph{Picard
groupoid}.
\end{definition}

The definition of braided categorical groups goes back to Joyal and\ Street
\cite[\S 3]{MR1250465}.

Given braided categorical groups $\mathsf{P},\mathsf{P}^{\prime}$, we write
$\operatorname*{Hom}^{\otimes}(\mathsf{P},\mathsf{P}^{\prime})$ for the
category of braided monoidal functors from $\mathsf{P}$ to $\mathsf{P}%
^{\prime}$: Objects are functors $\mathsf{P}\rightarrow\mathsf{P}^{\prime}$
preserving the monoidal structure along with the braiding and associativity
constraints; morphisms (usually called \textquotedblleft
homotopies\textquotedblright) are natural transformations between such
functors, which however also need to be compatible with the braided monoidal
structure. A fully spelled out definition of braided monoidal functors is
given for example in \cite[\S 6, take $G=1$]{MR2293318}.

\begin{remark}
[Notation]Some variation in language exists. In Deligne's work \cite[\S 4.1]%
{MR902592} Picard groupoids are called `commutative Picard categories', while
in \cite[\S 2.1]{MR2842932} merely `Picard categories'. Monoidal categories
are often called tensor categories instead.
\end{remark}

\subsection{General skeletalization and strictification methods}

A braided categorical group $(\mathsf{C},\otimes)$ has three key invariants:

\begin{enumerate}
\item $\pi_{0}(\mathsf{C},\otimes)$, the set of isomorphism classes of
objects; an abelian group,

\item $\pi_{1}(\mathsf{C},\otimes)$, the automorphism group of the unit object
$\mathbf{1}_{\mathsf{C}}$; another abelian group,

\item the map $q:\pi_{0}(\mathsf{C},\otimes)\rightarrow\pi_{1}(\mathsf{C}%
,\otimes)$ which assigns to any object $X\in\mathsf{C}$ its signature%
\[
q([X]):=s_{X,X}\otimes X^{-1}\otimes X^{-1}\text{,}%
\]
which is an automorphism of the unit $\mathbf{1}_{\mathsf{C}}$. Here $s_{X,X}$
denotes the braiding%
\[
s_{X,X}:X\otimes X\overset{\sim}{\longrightarrow}X\otimes X\text{.}%
\]
This is an automorphism of $X\otimes X$, and along the contractions
$\varepsilon_{X}$ of Equation \ref{lefmu7a} induces one of $\mathbf{1}%
_{\mathsf{C}}$, thus giving an element of $\pi_{1}(\mathsf{C},\otimes)$.
\end{enumerate}

There are particularly nice types of braided categorical groups:

\begin{definition}
A category $\mathsf{C}$ is called \emph{skeletal} if for any isomorphic
objects $X,X^{\prime}\in\mathsf{C}$ we must have $X=X^{\prime}$, i.e.%
\[
X\simeq X^{\prime}\qquad\Rightarrow\qquad X=X^{\prime}.
\]

\end{definition}

\begin{definition}
A braided categorical group $(\mathsf{C},\otimes)$ is called \emph{strict}
(but we usually write \emph{strictly associative}\footnote{Sometimes people
use the word \textquotedblleft strict\textquotedblright\ also to imply that
the commutativity constraint would have to be trivial, which would be much
more restrictive.}) if the associativity constraint%
\[
a_{X,Y,Z}:X\otimes(Y\otimes Z)\overset{\sim}{\longrightarrow}(X\otimes
Y)\otimes Z
\]
and the unit constraints%
\[
\mathbf{1}_{\mathsf{C}}\otimes X\overset{\sim}{\longrightarrow}X\qquad
\text{and}\qquad X\otimes\mathbf{1}_{\mathsf{C}}\overset{\sim}{\longrightarrow
}X
\]
are all identity maps (and in particular the objects on the left and right of
these isomorphisms are the same). If the braiding is symmetric, one also calls
such a $(\mathsf{C},\otimes)$ a \emph{permutative category}.
\end{definition}

Both of these properties do not really occur in nature much.

\begin{example}
Let $k$ be a field. The category $(\mathsf{Vect}_{f}(k),\otimes)$ of
$k$-vector spaces with the usual tensor product has neither property. In
particular, if one honestly follows the definitions,%
\[
k^{\ell}\otimes(k^{n}\otimes k^{m})\overset{\sim}{\longrightarrow}(k^{\ell
}\otimes k^{n})\otimes k^{m}%
\]
is an isomorphism between two genuinely different objects.
\end{example}

This example is too basic, and in many ways not really helpful. The
associativity constraint does such a truly basic and light thing in this
example that it is really hard to imagine that anything could ever go wrong or
that it is truly worth keeping the associativity constraint in mind.

As this is a critical possible misconception, we shall dwell on this for a
bit, even though every category theory inclined reader may shake their head in
boredom. We shall look at some non-trivial examples soon.

But first, we point out two basic simplication procedures for arbitrary
braided categorical groups:

\begin{theorem}
\label{thm_EverySymmMonoidalSkeletalizable}Every braided categorical group is
braided monoidal equivalent to a skeletal one. We call this
\emph{skeletalization}.
\end{theorem}

\begin{theorem}
[Mac Lane--Isbell]\label{thm_EverySymmMonoidalStrictifiable}Every braided
monoidal category is braided monoidal equivalent to a strictly associative
one. We call this \emph{strictification}.
\end{theorem}

We shall sketch both proofs below, in order to see what can go wrong. However,
aside from proving these, and this is crucial, in general it is
\textit{absolutely impossible} to achieve being strictly associative and
skeletal simultaneously.

\begin{proof}
[Proof sketch for Theorem \ref{thm_EverySymmMonoidalSkeletalizable}]Any
category is equivalent to a skeletal one: Given $\mathsf{C}$, (as Step 1) pick
precisely one object $X$ from any isomorphism class. Further, for any object
in $\mathsf{C}$, pick (as Step 2) a fixed isomorphism to the single chosen
object representing its isomorphism class. Now consider the full subcategory
$\mathsf{C}^{\prime}$ of these objects. It is, by construction, skeletal. The
inclusion $\mathsf{C}^{\prime}\rightarrow\mathsf{C}$ is clearly a fully
faithful functor, but it is also essentially surjective by construction. Now
we make $\mathsf{C}^{\prime}$ braided monoidal: For any $X,Y\in\mathsf{C}%
^{\prime}$ we have $X\otimes_{\mathsf{C}}Y$ in $\mathsf{C}$, and there will be
precisely one object $C_{X,Y}\in\mathsf{C}^{\prime}$ in the same isomorphism
class (by Step 1). Define $X\otimes_{\mathsf{C}^{\prime}}Y:=C_{X,Y}$. Define
the associativity and commutativity contraint by pre- and postcomposing the
ones in $\mathsf{C}$ with the fixed isomorphisms (of Step 2) to the fixed
objects representing the isomorphism classes of $\mathsf{C}$ uniquely in
$\mathsf{C}^{\prime}$.
\end{proof}

We gave the proof because it makes clear how this procedure can mess up the
associativity and commutativity constraints, even if they were easy maps in
the original category.

\begin{example}
In the notation of the proof, given $X,Y\in\mathsf{C}$, we obtain objects
$C_{X,Y}$ and $C_{Y,X}$ in $\mathsf{C}^{\prime}$ and the braiding becomes%
\begin{equation}
C_{X,Y}\overset{\sim}{\longrightarrow}X\otimes Y\overset{\sim}{\underset
{s_{X,Y}}{\longrightarrow}}Y\otimes X\overset{\sim}{\longleftarrow}%
C_{Y,X}\text{,} \label{lumb1}%
\end{equation}
where $s_{X,Y}$ is the original braiding of $\mathsf{C}$. Since our category
$\mathsf{C}^{\prime}$ is skeletal, we must have $C_{X,Y}=C_{Y,X}$, so the
composite map is just an automorphism of $C_{X,Y}$. However, since we have no
overall control of the two outer isomorphisms in Equation \ref{lumb1} (the
ones chosen in Step 2 of the proof sketch above), it is by no means clear that
this is the identity automorphism.
\end{example}

But one may believe that a more careful construction in the proof would be
able to solve this problem. This is not so. Although well-documented in the
literature, let us have a look at Isbell's highly instructive counterexample,
which immediately crushes all hope.

\begin{theorem}
[Isbell \cite{MR0249484}]Let $\mathsf{C}$ be a category which has all
products. Suppose there is some object $X$ such that $X\times X\simeq X$ and
$\operatorname*{End}(X)$ has at least two elements. Make $\mathsf{C}$ a
symmetric monoidal category using the product, i.e. $X\otimes Y$ is a product
of $X,Y$. Then $(\mathsf{C},\otimes)$ is not monoidal equivalent to any
simultaneously strictly associative and skeletal monoidal category.
\end{theorem}

\begin{example}
The category of sets or the category of all $k$-vector spaces are examples.
The category of all finite-dimensional $k$-vector spaces is not an example
(see Example \ref{ex_FinDimVSpaces}).
\end{example}

\begin{proof}
For a product of objects $X,Y$, we write $p_{X}^{X,Y}$ resp. $p_{Y}^{X,Y}$ for
the two projections. Let $f,g,h:X\rightarrow X$ be arbitrary morphisms. We
consider the following four composite morphisms:%
\begin{equation}%
\xymatrix@!R=0.03in{
(X\times X) \times X \ar[r]^-{p_{X\times X}^{X\times X,X}} & X\times
X \ar[r]^-{f \times g} & X\times X \\
(X\times X) \times X \ar[r]^-{(f \times g) \times h} & (X\times X)\times
X \ar[r]^-{p_{X\times X}^{X\times X,X}} & X \times X \\
X\times(X\times X) \ar[r]^-{f \times(g \times h)} & X\times(X\times
X) \ar[r]^-{{p_{X}^{X,X\times X}}} & X \\
X\times(X\times X) \ar[r]^-{p_{X}^{X,X\times X}} & X \ar[r]^-{f} & X
}
\label{lixfi1}%
\end{equation}
Note that the first and the second composite morphism agree on the nose.
Moreover, the third and the fourth composite morphism agree. Proof by
contradiction: Assume $\mathsf{C}$ is monoidal equivalent to a strictly
associative skeletal category $\mathsf{C}^{\prime}$, say by some functor $F$.
We apply this functor and consider the four morphisms above. For simplicity,
let us from now on write $X$ instead of $F(X)$; or equivalently without loss
of generality assume $\mathsf{C}$ had been $\mathsf{C}^{\prime}$ to start
with. Two new things happen: Since our target category is strictly
associative, the associativity constraint $X\times(X\times X)\longrightarrow
(X\times X)\times X$ is between the same objects and the identity map. Since
the monoidal bifunctor is natural in each variable, it follows that the second
and third composite morphism simplify to%
\begin{align*}
&  X\times X\times X\overset{f\times g\times h}{\longrightarrow}X\times
X\times X\overset{p_{X\times X}^{X\times X,X}}{\longrightarrow}X\times X\\
&  X\times X\times X\overset{f\times g\times h}{\longrightarrow}X\times
X\times X\overset{p_{X}^{X,X\times X}}{\longrightarrow}X
\end{align*}
and the left arrows in both lines agree. However, since the category is
skeletal, $X\times X\cong X$ means that these are the same object. This means
that all first factor projections of the products here agree with $p_{X}%
^{X,X}$, using that our equivalence of categories has preserved products.
Thus, also the second arrows in both lines above agree. But this means that
the composites agree. Hence, all four morphisms in Equation \ref{lixfi1} are
the same. Moreover $X\times X\times X=X$. Thus, comparing the first with the
last composite in Equation \ref{lixfi1}, we obtain%
\[
f\circ p_{X}^{X,X}=(f\times g)\circ p_{X}^{X,X}%
\]
and this is valid for arbitrary choices of $f,g,h$ above. As $p_{X}^{X,X}$ is
an epic, we deduce $f=f\times g$. Swapping the r\^{o}les of $f,g$, we also
obtain $f\times g=g$, and then $f=g$ for arbitrary $f,g$, contradicting that
$f,g\in\operatorname*{End}(X)$ can be chosen different by assumption.
\end{proof}

\begin{example}
\label{ex_FinDimVSpaces}The category of all finite-dimensional $k$-vector
spaces with the direct sum $(\mathsf{Vect}_{f}(k),\oplus)$ is symmetric
monoidal equivalent to a both strictly associative and skeletal symmetric
monoidal category. As objects for our strictly associative skeletal
$\mathsf{C}^{\prime}$ take $\mathbb{Z}_{\geq0}$ and define $n\boxplus m:=n+m$,
$\operatorname*{Hom}\nolimits_{\mathsf{C}^{\prime}}(n,m):=\operatorname*{Hom}%
\nolimits_{k}(k^{n},k^{m})$ with the identity commutativity constraint. Then
$F:(\mathsf{C}^{\prime},\boxplus)\rightarrow(\mathsf{Vect}_{f}(k),\oplus)$
sending $n$ to $k^{n}$ is a symmetric monoidal equivalence. One can try to
extend this to allowing countably infinite dimension. As objects take
$\mathbb{Z}_{\geq0}\cup\{\infty\}$ and define $n\boxplus\infty:=\infty$ and
$\infty\boxplus\infty:=\infty$, and take $F(\infty):=k^{\oplus\infty}$. Given
$f,g\in\operatorname*{Hom}\nolimits_{k}(k^{\infty},k^{\infty})$, it is unclear
how to define
\[
f\oplus g\in\operatorname*{Hom}\nolimits_{k}(k^{\infty},k^{\infty})\text{,}%
\]
leading us back to Isbell's counterexample.
\end{example}

Now, let us quickly look at how to make a braided monoidal category strictly associative.

\begin{proof}
[Proof sketch for Theorem \ref{thm_EverySymmMonoidalStrictifiable}]Let $X$ be
the free monoid generated by the objects of $\mathsf{C}$, i.e. objects are
finite formal words%
\[
\text{\textquotedblleft}x_{1}x_{2}\ldots x_{m}\text{\textquotedblright,}%
\]
where the letters \textquotedblleft$x_{1},\ldots,x_{m}$\textquotedblright\ are
objects $x_{i}\in\mathsf{C}$. The empty word is also allowed. We define a
category $\mathsf{C}^{\prime}$ by having $X$ as its objects, and define on
objects%
\begin{align*}
\psi:\mathsf{C}^{\prime}  &  \longrightarrow\mathsf{C}\\
\left.  \text{\textquotedblleft}x_{1}x_{2}\ldots x_{m}\text{\textquotedblright%
}\right.   &  \longmapsto(\ldots(x_{1}\otimes x_{2})\otimes\ldots)\otimes
x_{m-1})\otimes x_{m})\text{,}%
\end{align*}
i.e. we apply the monoidal bifunctor to \textquotedblleft$x_{1}x_{2}\ldots
x_{m}$\textquotedblright, bracketed all to the left. The empty word goes to
$\mathbf{1}_{\mathsf{C}}$, a single letter \textquotedblleft$x_{1}%
$\textquotedblright\ to the object it represents, \textquotedblleft$x_{1}%
x_{2}$\textquotedblright\ to $x_{1}\otimes x_{2}$, \textquotedblleft%
$x_{1}x_{2}x_{3}$\textquotedblright\ to $((x_{1}\otimes x_{2})\otimes x_{3})$
and so on. We equip $\mathsf{C}^{\prime}$ with the structure of a category by
demanding that the above should be a fully faithful functor. This means that%
\begin{equation}
\operatorname*{Hom}\nolimits_{\mathsf{C}^{\prime}}(\text{\textquotedblleft%
}x_{1}x_{2}\ldots x_{m}\text{\textquotedblright},\text{\textquotedblleft}%
y_{1}y_{2}\ldots y_{n}\text{\textquotedblright}):=\operatorname*{Hom}%
\nolimits_{\mathsf{C}}(\psi(x_{1}x_{2}\ldots x_{m}),\psi(y_{1}y_{2}\ldots
y_{n}))\text{.} \label{lefmu1}%
\end{equation}
Define a functor $\psi^{\prime}:\mathsf{C}\rightarrow\mathsf{C}^{\prime}$ in
the opposite direction by sending the object $x$ to the single letter word
\textquotedblleft$x$\textquotedblright. This gives an equivalence of
categories. Define a braided monoidal structure on $\mathsf{C}^{\prime}$ by
simply concatenating words,%
\[
\text{\textquotedblleft}x_{1}x_{2}\ldots x_{m}\text{\textquotedblright}%
\otimes\text{\textquotedblleft}y_{1}y_{2}\ldots y_{n}\text{\textquotedblright%
}:=\text{\textquotedblleft}x_{1}x_{2}\ldots x_{m}y_{1}y_{2}\ldots
y_{n}\text{\textquotedblright,}%
\]
and compatibly on morphisms (as the $\operatorname*{Hom}$-sets agree with the
input category $\mathsf{C}$, one can import this part of the structure, see
Equation \ref{lefmu1}). Let the empty word \textquotedblleft\textquotedblright%
\ act as the tensor unit $\mathbf{1}_{\mathsf{C}^{\prime}}$. Since we really
just concatenate words, it is clear that%
\[
\mathbf{1}_{\mathsf{C}^{\prime}}\otimes\text{\textquotedblleft}x_{1}%
x_{2}\ldots x_{m}\text{\textquotedblright}=\text{\textquotedblleft}x_{1}%
x_{2}\ldots x_{m}\text{\textquotedblright,}%
\]
and correspondingly for \textquotedblleft$x_{1}x_{2}\ldots x_{m}%
$\textquotedblright$\otimes\mathbf{1}_{\mathsf{C}^{\prime}}$, so
$(\mathsf{C}^{\prime},\otimes)$ indeed is a strictly associative braided
monoidal category. Finally, one has to check that the equivalence of
categories $\psi$ is braided monoidal. See \cite{MR1250465} for a complete argument.
\end{proof}

We see that the above procedure increases the cardinality of pairwise
isomorphic objects massively.

\begin{example}
If we strictify associativity in $(\mathsf{Vect}_{f}(k),\oplus)$ using the
procedure of the above proof, objects will be finite words of vector spaces
\textquotedblleft$x_{1}x_{2}\ldots x_{m}$\textquotedblright\ which end up
being isomorphic if and only if $\sum\dim x_{i}=\sum\dim x_{i}^{\prime}$.
Thus, we get a very different category from the strictly associative one in
Example \ref{ex_FinDimVSpaces}.
\end{example}

We also see that alternating between this strictification and skeletalization
procedure will just yield more and more complicated categories, which rather
carry us farther away from our goal of having a simultaneously skeletal and
strictly associative model.

\section{Strictification\label{sect_Strictification}}

\subsection{Abelian and symmetric cohomology}

In order to dig a little deeper into these problems, we should first recall
two `cohomology theories' (of sorts). We first need to recall how group
cohomology can be defined, from two different viewpoints. Let $G$ be a group
and $M$ a $G$-module. In pure algebra, \emph{group cohomology} is usually
defined as%
\[
H_{grp}^{n}(G,M):=\operatorname*{Ext}\nolimits_{\mathbb{Z}[G]}^{n}%
(\mathbb{Z},M)
\]
in the category of $\mathbb{Z}[G]$-modules, where $\mathbb{Z}[G]$ is the group
ring. Since $\operatorname*{Ext}\nolimits_{\mathbb{Z}[G]}^{0}(\mathbb{Z}%
,M)=\operatorname*{Hom}\nolimits_{\mathbb{Z}[G]}(\mathbb{Z},M)=M^{G}$, these
$\operatorname*{Ext}$-groups correspond to the right derived functor groups
$\mathbf{R}^{n}(-)^{G}$ of the left-exact functor of taking $G$-invariants. In
topology one may prefer a different take on the same concept: Write $K(M,n)$
for the $n$-th Eilenberg--Mac\thinspace Lane space with single non-zero
homotopy group $M$ in degree $n$. If $M$ has the trivial $G$-module structure,
one can also define group cohomology as%
\[
H_{grp}^{n}(G,M):=[K(G,1),K(M,n)]\text{,}%
\]
the pointed homotopy classes of maps from $K(G,1)$ (which is nothing but the
classifying space of the group).

The topological as well as the algebraic approach to the definition are
equivalent and give the same cohomology groups. While we mostly work with an
algebraic viewpoint below, the topologist's definition makes it much easier to
motivate the construction of two other cohomology theories:

\begin{definition}
[Eilenberg--Mac\thinspace Lane]\label{def_EMStyleTop}Let $G$ be an abelian
group and $M$ an abelian group (regarded as a trivial $G$-module).

\begin{enumerate}
\item The groups%
\[
H_{ab}^{n}(G,M):=[K(G,2),K(M,n+1)]
\]
are called \emph{abelian cohomology}.

\item The groups%
\[
H_{sym}^{n}(G,M):=[K(G,3),K(M,n+2)]
\]
are called \emph{symmetric cohomology}.
\end{enumerate}
\end{definition}

We can quickly motivate where these definitions originate from: There are
equivalences of homotopy categories, due to Joyal--Tierney (following ideas of
Grothendieck),%
\begin{align}
Ho(\text{categorical groups})  &  \longrightarrow Ho(\text{homotopy
}[1,2]\text{-types})\text{,}\nonumber\\
Ho(\text{braided categorical groups})  &  \longrightarrow Ho(\text{homotopy
}[2,3]\text{-types})\text{,}\label{lzx1}\\
Ho(\text{Picard groupoids})  &  \longrightarrow Ho(\text{homotopy
}[n,n+1]\text{-types})\text{,}\nonumber\\
&  \qquad\qquad\qquad\text{(for any }n\geq3\text{)}\nonumber
\end{align}
where on the right we refer to pointed unstable homotopy types. In each case
the objects on the right-hand side are determined by providing the groups
$G,M$ forming the two consecutive non-trivial $\pi_{n},\pi_{n+1}$ as well as a
cocycle corresponding to the corresponding Postnikov invariant, and concretely
the relevant $k$-invariant lies in%
\[
H^{3}(G,M)\text{,\quad resp.}\quad H_{ab}^{3}(G,M)\text{,\quad resp.}\quad
H_{sym}^{3}(G,M)\text{,}%
\]
depending on which of the three cases one considers. In Equation \ref{lzx1}
the functor inducing the equivalence is always a suitable kind of nerve. We
refer to \cite{MR1219923}, \cite{MR1414569}, \cite{MR1890928},
\cite{MR2981952} for details on this story. We shall see this in a little more
detail in \S \ref{sect_JoyalStreetClassif} in the braided case, which is the
only instance of the above relevant for our purposes.

While introducing the cohomology groups of Definition \ref{def_EMStyleTop} in
the above topological language is nice conceptually, it does not help much for
computations. For group cohomology one can use the standard
Chevalley--Eilenberg projective resolution of the trivial $G$-module
$\mathbb{Z}$ as a $\mathbb{Z}[G]$-module in order to compute
$\operatorname*{Ext}\nolimits_{\mathbb{Z}[G]}^{n}(\mathbb{Z},M)$. When using
the quasi-isomorphic complex of normalized cocycles $P_{\bullet}$ (so that the
complex%
\[
(\operatorname*{Hom}\nolimits_{\mathbb{Z}[G]}(P_{\bullet},M),\partial
^{\bullet})
\]
is the one which is frequently called the complex of `inhomogeneous chains',
e.g., \cite[Chapter I, \S 2]{MR2392026}) instead, a \emph{group }%
$3$\emph{-cocycle} is a map%
\[
h:G^{3}\longrightarrow M
\]
such that $\partial h=0$, which unravels as the condition%
\begin{equation}
h(x,y,z)+h(u,x+y,z)+h(u,x,y)=h(u,x,y+z)+h(u+x,y,z)\text{.} \label{lefmu1a}%
\end{equation}
A \emph{group }$3$\emph{-coboundary} is $\partial k$ for some $k:G^{2}%
\rightarrow M$ such that $h$ has the shape%
\begin{equation}
h(x,y,z)=k(y,z)-k(x+y,z)+k(x,y+z)-k(x,y)\text{.} \label{lefmu2}%
\end{equation}
These explicit expressions can directly be unravelled from \cite[Chapter I,
\S 2]{MR2392026} for example.

In the literature on symmetric or braided monoidal categories, one often uses
the following additional condition:

\begin{definition}
\label{def_NormalizedCochain}An inhomogeneous chain $h:G^{n}\rightarrow M$ is
called \emph{normalized} if $h(x_{1},\ldots,x_{n})=0$ as soon as $x_{i}=0$ for
some $i$.
\end{definition}

One can always restrict to considering normalized cochains in view of the
following fact.

\begin{lemma}
\label{lemma_CochainsCanBeAssumedNormalized}Given an inhomogeneous chain
$h:G^{n}\rightarrow M$, there exists a normalized inhomogeneous chain
$h^{\prime}:G^{n}\rightarrow M$ such that $[h]\equiv\lbrack h^{\prime}]\in
H^{n}(G,M)$ both represent the same group cohomology class.
\end{lemma}

\begin{proof}
This corresponds to using the normalized versus the unnormalized bar complex,
see \cite[Chapter 6, \S 6.5.5]{MR1269324} for background. For a direct proof
without going back to the formalism of bar complexes, one can follow
\cite[Chapter I, \S 2, Exercise 5]{MR2392026}.
\end{proof}

Next, following Eilenberg and Mac\thinspace Lane we discuss the analogous
explicit expressions for abelian and symmetric cohomology.

\begin{definition}
\label{def_Ab3Cocycle}Let $G,M$ be abelian groups.

\begin{enumerate}
\item An \emph{abelian }$3$\emph{-cocycle} is a pair $(h,c)$ consisting of a
group $3$-cocycle $h:G^{3}\rightarrow M$ such that%
\begin{equation}
h(x,0,z)=0 \label{lefmu2aa}%
\end{equation}
and a map $c:G^{2}\rightarrow M$ which satisfies%
\begin{align}
h(y,z,x)+c(x,y+z)+h(x,y,z)  &  =c(x,z)+h(y,x,z)+c(x,y)\tag{A}\label{lx_a_1}\\
-h(z,x,y)+c(x+y,z)-h(x,y,z)  &  =c(x,z)-h(x,z,y)+c(y,z) \tag{A'}\label{lx_a_2}%
\end{align}
for all $x,y,z\in G$.

\item An \emph{abelian }$3$\emph{-coboundary}\textit{ is a group }$3$-cocycle
$h=\partial k$ for some map $k:G^{2}\rightarrow M$ with $k(x,0)=0$ and
$k(0,y)=0$. It can be interpreted as an abelian $3$-cocycle $(h,c)$ using%
\begin{equation}
c(x,y):=k(x,y)-k(y,x)\text{.} \label{lefmu2a}%
\end{equation}

\end{enumerate}
\end{definition}

Note that every abelian $3$-coboundary is indeed an abelian $3$-cocycle, e.g.%
\[
h(x,0,z)=k(0,z)-k(x,0)=0
\]
holds by Equation \ref{lefmu2}.

\begin{remark}
The condition $h(x,0,z)=0$ in Equation \ref{lefmu2a} implies $h(x,y,z)=0$ as
soon as one of $x,y,z$ is zero. To see this, use Equation \ref{lefmu1a} with
$y=0$, giving $h(u,x,0)=0$ and then Equation \ref{lefmu1a} with $x=0$. This
shows that these cocycles are automatically assumed normalized in the sense of
Definition \ref{def_NormalizedCochain}.
\end{remark}

\begin{definition}
\label{def_SymmetricCocycles}A \emph{symmetric }$3$\emph{-cocycle} is an
abelian $3$-cocycle $(h,c)$ such that%
\begin{equation}
c(x,y)=-c(y,x) \label{lx_b_1}%
\end{equation}
for all $x,y\in G$. A \emph{symmetric }$3$\emph{-coboundary} is the same as an
abelian $3$-coboundary.
\end{definition}

Note that an abelian $3$-coboundary is indeed always a symmetric $3$-cocycle,
as is seen from Equation \ref{lefmu2a}.

\begin{definition}
A \emph{quadratic form} $q:G\rightarrow M$ is any map of sets such that
$q(x)=q(-x)$ and%
\begin{equation}
b(x,y):=q(x+y)-q(x)-q(y) \label{lefmu2b}%
\end{equation}
is $\mathbb{Z}$-bilinear for all $x,y\in G$. This bilinear form is called the
\emph{polarization} \emph{form}. Write $\operatorname*{Quad}(G,M)$ for the set
of all quadratic forms.
\end{definition}

\begin{theorem}
[Eilenberg--Mac Lane]\label{thm_EMDeepIso}The group $H_{ab}^{3}(G,M)$ (resp.
$H_{sym}^{3}(G,M)$) of Definition \ref{def_EMStyleTop} can be described as the
set of abelian (resp. symmetric) $3$-cocycles modulo abelian (resp. symmetric)
$3$-coboundaries in the sense of Definition \ref{def_Ab3Cocycle} (resp.
Definition \ref{def_SymmetricCocycles}). The map%
\begin{align}
\operatorname*{tr}:H_{ab}^{3}(G,M)  &  \longrightarrow\operatorname*{Quad}%
(G,M)\label{lEMMa2}\\
(h,c)  &  \longmapsto(x\mapsto c(x,x))\nonumber
\end{align}
is an isomorphism.
\end{theorem}

We refer to \cite{MR0045115} for an overview. The paper \cite{MR0056295} is
essentially entirely devoted to making the type of translation underlying the
above theorem.

\subsection{Joyal--Street classification\label{sect_JoyalStreetClassif}}

We quickly summarize the Joyal--Street classification. Write $\mathcal{Q}uad$
for the $1$-category (a) whose objects are triples $(G,M,q)$ with $G,M$
abelian groups and $q\in\operatorname*{Quad}(G,M)$, and (b) morphisms
$(G,M,q)\rightarrow(G^{\prime},M^{\prime},q^{\prime})$ are pairs of group
homomorphisms $f:G\rightarrow G^{\prime}$ and $g:M\rightarrow M^{\prime}$ such
that the square%
\[%
\xymatrix{
G \ar[r]^{f} \ar[d]_{q} & G^{\prime} \ar[d]^{q^{\prime}} \\
M \ar[r]_{g} & M^{\prime}
}%
\]
commutes. Write $\mathcal{BCG}$ for the $2$-category (a) whose objects are
braided categorical groups and (b) whose arrows are braided monoidal functors
and (c) whose $2$-arrows are braided monoidal natural equivalences of
functors. For details on the precise definition we would perhaps recommend the
nice and very careful treatment of Cegarra--Khmaladze \cite[\S 6]{MR2293318}.
They treat the $G$-graded case, so for $G=1$ the trivial group, their
description specializes to the concepts we use here. We write
$Ho(\mathcal{BCG})$ for what one could call the homotopy category of braided
categorical groups\ (or $1$-category truncation); its objects are braided
categorical groups and morphisms are the equivalence classes of braided
monoidal functors. As homotopy (or $2$-equivalence) preserves composition,
$Ho(\mathcal{BCG})$ is a well-defined $1$-category. There is a functor%
\[
T:Ho(\mathcal{BCG})\longrightarrow\mathcal{Q}uad\text{,}\qquad\qquad
(\mathsf{C},\otimes)\mapsto(\pi_{0}\mathsf{C},\pi_{1}\mathsf{C},q)\text{,}%
\]
where $q$ is the quadratic form attached to the abelian $3$-cohomology class
encoding the associator and braiding, using Equation \ref{lEMMa2}.

\begin{theorem}
[Joyal--Street]\label{thm_JSClassif}The functor $T$ has the following properties.

\begin{enumerate}
\item It is well-defined.

\item It is essentially surjective.

\item It is conservative, i.e. $F:(\mathsf{C},\otimes)\rightarrow
(\mathsf{C}^{\prime},\otimes^{\prime})$ is a braided monoidal equivalence if
and only if $TF$ is an isomorphism in $\mathcal{Q}uad$.

\item It is full, i.e. any morphism of $\mathcal{Q}uad$ comes from a (homotopy
class of) braided monoidal functor(s) of braided categorical groups.
\end{enumerate}
\end{theorem}

See \cite[Theorem 3.3 and Remark 3.3]{MR1250465} or \cite[page 435, again for
$G=1$ and the category is called $\mathcal{H}_{1,ab}^{3}$ instead]{MR2293318}
or \cite[\S 6-7]{joyalstreetpreprint}. S\'{\i}nh had earlier proven an
analogous result in the case of symmetric braidings, i.e. Picard groupoids,
\cite{sinh}.

We can be a bit more precise: The essential surjectivity is proven by
providing an explicit skeletal model with the given invariants:

\begin{definition}
Given abelian groups $G,M$ and an abelian $3$-cocycle $(h,c)$, let
$\mathcal{T}(G,M,(h,c))$ be the following braided categorical group:

\begin{enumerate}
\item objects are elements in $G$,

\item automorphisms of any object $X\in G$ are $\operatorname*{Aut}(X):=M$,

\item there are no morphisms except for automorphisms, and their composition
is addition in $M$,

\item the monoidal structure is%
\[
(X\overset{f}{\longrightarrow}X)\otimes(X^{\prime}\overset{f^{\prime}%
}{\longrightarrow}X^{\prime}):=(X+X^{\prime}\overset{f+f^{\prime}%
}{\longrightarrow}X+X^{\prime})\text{,}%
\]
where addition is just addition in $G$ (on objects) resp. in $M$ (for
$f,f^{\prime}$),

\item the associator%
\[
a_{X,Y,Z}:X\otimes(Y\otimes Z)\overset{\sim}{\longrightarrow}(X\otimes
Y)\otimes Z
\]
is just the automorphism defined by $h(X,Y,Z)\in M$,

\item and analogously for the commutativity constraint and $c(X,Y)$.
\end{enumerate}
\end{definition}

As implied by Theorem \ref{thm_JSClassif} and the remark about how to prove
essential surjectivity, any braided categorical group $(\mathsf{C},\otimes)$
is braided monoidal equivalent to the skeletal braided categorical group
$\mathcal{T}(\pi_{0}\mathsf{C},\pi_{1}\mathsf{C},(h,c))$ for some $[(h,c)]\in
H_{ab}^{3}(\pi_{0}\mathsf{C},\pi_{1}\mathsf{C})$, where $[(h,c)]$ is
well-defined up to abelian $3$-coboundaries. Moreover, this detects braided
monoidal equivalence: A braided monoidal functor $F:(\mathsf{C},\otimes
)\rightarrow(\mathsf{C}^{\prime},\otimes^{\prime})$ is an equivalence if and
only if it induces isomorphisms on $\pi_{0}$, $\pi_{1}$ and under these the
cohomology class gets identified. For Picard groupoids, these results were
established earlier by S\'{\i}nh. The functor $T$ can also be used to describe
various unstable homotopy types combinatorially by combining it with Equation
\ref{lzx1}. See \cite{MR1096295} for more on this.

\begin{remark}
\label{rmk_StrictlyAssociativeAndSkeletal}We see that a skeletal braided
categorical group is braided monoidal equivalent to a skeletal strictly
associative one if and only if the cohomology class $[(h,c)]$ has an abelian
$3$-cocycle representative with $h(x,y,z)=0$ for all $x,y,z\in G$.
\end{remark}

\begin{proof}
Just use that any skeletal braided categorical group $\mathsf{C}$ is
necessarily of the form $\mathcal{T}(\pi_{0}\mathsf{C},\pi_{1}\mathsf{C}%
,(h,c))$ itself.
\end{proof}

\section{Polar forms\label{sect_Polar}}

\subsection{Polar quadratic forms}

\begin{definition}
\label{def_PolarQuadraticMap}We call a quadratic form \emph{polar} (or
\emph{polar for} $t$) if its polarization is of the shape%
\[
b(x,y)=t(x,y)+t(y,x)
\]
for some bilinear form $t$ (which is not required to be symmetric or non-degenerate).
\end{definition}

As we shall see in Lemma \ref{lemma_CharPolar}, this definition uniformly
characterizes two somewhat different sources of examples in one.

\begin{example}
\label{example_Polar1}Suppose $t:G\times G\rightarrow M$ is a $\mathbb{Z}%
$-bilinear form, not necessarily symmetric. Then $q(x):=t(x,x)$ is a polar
quadratic form by direct computation,%
\[
b(x,y)=t(x,y)+t(y,x)\text{.}%
\]

\end{example}

\begin{example}
\label{example_Polar2}Suppose $s:G\rightarrow M$ is any group homomorphism to
an $\mathbb{F}_{2}$-vector space $M$. Then $q(x):=s(x)$ is a polar quadratic
form because $b(x,y)=q(x+y)-q(x)-q(y)=0$.
\end{example}

\begin{example}
Consider $q(x):=x^{2}$ as a map $\mathbb{Z}\rightarrow\mathbb{Z}$. This
restricts to a well-defined map $\mathbb{Z}/2\rightarrow\mathbb{Z}/4$ in view
of%
\[
(x+2m)^{2}=x^{2}+4mx+4m^{2}\equiv x^{2}\operatorname{mod}4\text{.}%
\]
Thus, for $G:=\mathbb{Z}/2$ and $M:=\mathbb{Z}/4$ we obtain a quadratic form,
$q(-x)=q(x)$, with the bilinear symmetric polarization $b(x,y):=2xy$. We claim
that this form is \emph{not} polar. Assume it were. We must have
$t:\mathbb{F}_{2}\times\mathbb{F}_{2}\rightarrow\mathbb{Z}/4\mathbb{Z}$ and
$t(x,y)=xyn$ for some $n\in\mathbb{Z}$ by bilinearity, and thus without loss
of generality $n=1$ as we need $b(x,y)=t(x,y)+t(y,x)$. However, $(x,y)\mapsto
xy$ is not bilinear as a map $\mathbb{F}_{2}\times\mathbb{F}_{2}%
\rightarrow\mathbb{Z}/4\mathbb{Z}$: We have%
\[
0=c(0,1)=c(1+1,1)\neq c(1,1)+c(1,1)=2\qquad\text{in}\qquad\mathbb{Z}/4\text{.}%
\]
On the other hand, define%
\[
h(x,y,z):=\left\{
\begin{array}
[c]{ll}%
2 & \text{for }x=y=z=1\\
0 & \text{else}%
\end{array}
\right.
\]
and $c(x,y):=xy$, which is just%
\[
c(x,y):=\left\{
\begin{array}
[c]{ll}%
1 & \text{for }x=y=z=1\\
0 & \text{else.}%
\end{array}
\right.
\]
Then $(h,c)$ is an abelian $3$-cocycle. We repeat that $c$ is not bilinear,
contrary to what the trained eye might misconceive.
\end{example}

\begin{lemma}
\label{lemma_QuadMapChar}An arbitrary map of sets $q:G\rightarrow M$ is
quadratic if and only if%
\begin{align}
q(-x)  &  =q(x)\label{lefmu2c}\\
q(x+y+z)+q(x)+q(y)+q(z)  &  =q(y+z)+q(z+x)+q(x+y) \label{lefmu2ca}%
\end{align}
both hold for all $x,y,z\in G$.
\end{lemma}

\begin{proof}
We demand $q(-x)=q(x)$ in both situations, so we only need to show that
Equation \ref{lefmu2ca} is equivalent to the $\mathbb{Z}$-bilinearity of
$b(x,y)$ in Equation \ref{lefmu2b}. Since the latter is visibly symmetric, it
suffices to check linearity in the first variable, that is, the vanishing of%
\[
b(x+y,z)-b(x,z)-b(y,z)
\]
for all $x,y,z$. When unravelling each $b(-,-)$ using Equation \ref{lefmu2b},
we obtain an expression which is zero (literally) if and only if Equation
\ref{lefmu2ca} holds.
\end{proof}

\begin{lemma}
If $(h,c)$ is an abelian $3$-cocycle, we have%
\[
\sum_{\sigma\in S_{3}}\operatorname*{sgn}(\sigma)h(x_{\sigma(1)},x_{\sigma
(2)},x_{\sigma(3)})=0\text{,}%
\]
i.e. the alternating sum over all permutations of the arguments is zero.
\end{lemma}

\begin{proof}
In Identity \ref{lx_a_1} bring all terms relying on $c$ on the left side,
giving%
\[
c(x,y+z)-c(x,y)-c(x,z)=-h(x,y,z)+h(y,x,z)-h(y,z,x)\text{.}%
\]
The left side of the equation is invariant under swapping $y$ and $z$, and
thus so must be the right side. Thus, their difference%
\[
-h(x,y,z)+h(y,x,z)-h(y,z,x)+h(x,z,y)-h(z,x,y)+h(z,y,x)=0
\]
must be zero. But this is what we had to prove.
\end{proof}

\begin{lemma}
\label{lemma_v2}If $(h,c)$ is an abelian $3$-cocycle, then
$W(x,y):=c(x,y)+c(y,x)$ is $\mathbb{Z}$-bilinear and symmetric.
\end{lemma}

\begin{proof}
Since $W$ is symmetric in $x,y$, it suffices to check linearity in $x$. We
find%
\[
W(x+z,y)-W(x,y)-W(z,y)=c(x+z,y)+c(y,x+z)-c(x,y)-c(y,x)-c(z,y)-c(y,z)
\]
and evaluating $c(x+z,y)$ using Identity \ref{lx_a_2} (and observing a number
of cancellations), we obtain%
\[
=h(x,z,y)-h(x,y,z)+h(y,x,z)+c(y,x+z)-c(y,x)-c(y,z)\text{.}%
\]
Next, unravel $c(y,x+z)$ using Identity \ref{lx_a_1}, giving%
\[
=h(x,z,y)-h(x,y,z)+h(y,x,z)-h(y,x,z)+h(x,y,z)-h(x,z,y)=0\text{,}%
\]
which vanishes literally (all terms appear twice, but with opposite signs).
\end{proof}

\begin{lemma}
\label{lemma_v3}If $(h,c)$ is an abelian $3$-cocycle, then we have%
\[
c(y+z,y+z)-c(y,y)-c(z,z)=W(y,z)
\]
for all $y,z$ and $W$ as in the previous lemma.
\end{lemma}

\begin{proof}
We again use Identity \ref{lx_a_1}, in the form%
\[
c(x,y+z)-c(x,y)-c(x,z)=-h(x,y,z)+h(y,x,z)-h(y,z,x)\text{.}%
\]
Now we plug in $x:=y+z$. This yields%
\[
c(y+z,y+z)-c(y+z,y)-c(y+z,z)=-h(y+z,y,z)+h(y,y+z,z)-h(y,z,y+z)
\]
and now use Identity \ref{lx_a_2} to unravel $c(y+z,y)$ and $c(y+z,z)$. We get%
\begin{align*}
&  c(y+z,y+z)-c(z,y)-c(y,y)-c(y,z)-c(z,z)\\
&  \qquad=-h(y+z,y,z)+h(y,y+z,z)-h(y,z,y+z)\\
&  \qquad\qquad+h(z,y,z)+h(y,z,y)\text{.}%
\end{align*}
Recall the group $3$-cocycle condition, Equation \ref{lefmu1a}, but plug in
$u:=y$ and $x:=z$. Then the cocycle condition tells us that the right side of
this equation vanishes. This proves our claim.
\end{proof}

\begin{lemma}
\label{lemma_TraceIsQuadratic}If $(h,c)$ is an abelian $3$-cocycle, then
$q(x):=c(x,x)$ is a quadratic form.
\end{lemma}

This quadratic form is called the \emph{trace} and underlies the map of
Equation \ref{lEMMa}. The above lemma is of course due to Eilenberg and
Mac\thinspace Lane, and quite old, but since almost all texts just refer to
this verification as an \textsl{Exercise}, we felt we should spell it out, if
only to provide a reference.

\begin{proof}
\textit{(Step 1)} We compute, just by unravelling $q$ in terms of $c$ and
using Lemma \ref{lemma_v3} three times, that $q(x+y)+q(y+z)+q(x+z)$ is%
\begin{align}
&  =c(x+y,x+y)+c(y+z,y+z)+c(x+z,x+z)\nonumber\\
&  =2c(x,x)+2c(y,y)+2c(z,z)+W(x,y)+W(y,z)+W(x,z)\text{.} \label{lmef1}%
\end{align}
Next, use Lemma \ref{lemma_v3} but plug in $x+y$ instead of $y$, giving%
\[
c(x+y+z,x+y+z)-c(x+y,x+y)-c(z,z)=W(x+y,z)\text{.}%
\]
Apply Lemma \ref{lemma_v3} to unravel $c(x+y,x+y)$ and use the bilinearity of
$W$ (Lemma \ref{lemma_v2}), giving%
\[
c(x+y+z,x+y+z)-c(x,x)-c(y,y)-c(z,z)=W(x,y)+W(x,z)+W(y,z)\text{.}%
\]
Now, we plug in Equation \ref{lmef1} for the $W$-terms in the above equation,
giving%
\begin{equation}
q(x+y+z)+q(x)+q(y)+q(z)=q(x+y)+q(y+z)+q(x+z)\text{.} \label{lmef2}%
\end{equation}
Thus, by Lemma \ref{lemma_QuadMapChar}, we are done once we show that
$q(x)=q(-x)$.\newline\textit{(Step 2) }To this end, use Equation \ref{lmef2}
with $y:=-x$, $z:=x$. We get%
\begin{equation}
3q(x)+q(-x)=2q(0)+q(2x)\text{.} \label{lmefa}%
\end{equation}
Lemma \ref{lemma_v3} with all variables zero shows $-q(0)=W(0,0)=0$ by the
bilinearity of $W$. Next, use Lemma \ref{lemma_v3} with $y=z$ and call the
variable $x$. We get $q(2x)-2q(x)=W(x,x)$. Combine all three formulas to
obtain $q(x)+q(-x)=W(x,x)$, but unravelling $W$ using its definition yields
$2q(x)$. Thus, we obtain $q(-x)=q(x)$, as desired.
\end{proof}

The following goes back to J. H. C. Whitehead \cite{MR35997}.

\begin{lemma}
\label{lemma_WhiteheadExactness}There is a commutative diagram%
\[%
\xymatrix{
& & \operatorname*{Hom}(G\otimes G,M) \ar[d]^{h} \ar[dr]^{\operatorname{sym}}
\\
0 \ar[r] & \operatorname*{Hom}(G/2G,M) \ar@{^{(}->}[r]^-{\psi} & \operatorname
*{Quad}(G,M) \ar[r]^-{\varphi} & \operatorname*{Hom}(G\otimes G,M) \\
0 \ar[r] & H^3_{sym}(G,M) \ar[u]^{\cong} \ar@{^{(}->}[r] & H^3_{ab}%
(G,M) \ar[u]^{\cong} \text{,} \\
}%
\]
where $\psi$ sends a map to itself and $\varphi$ sends $q$ to its
polarization, Equation \ref{lefmu2b}. The downward arrow $h$ sends a bilinear
map $B$ to $q(x):=B(x,x)$, and the diagonal arrow is the symmetrization map%
\[
b(x,y):=B(x,y)+B(y,x)\text{.}%
\]
The middle row in the diagram is exact.\ The upward arrows send $(h,c)$ to
$x\mapsto c(x,x)$.
\end{lemma}

\begin{proof}
The commutativity of the upper triangle is a one line computation. We check
the exactness of the middle row: If $f:G/2\rightarrow M$ is any $\mathbb{Z}%
$-linear map, we clearly have $b(x,y)=0$ in Equation \ref{lefmu2b} by
linearity and $f(a)=f(-a)$ holds since all elements in $G/2$ are $2$-torsion,
so $\psi$ is well-defined and injective (see also Example \ref{example_Polar2}%
). If $\varphi$ sends a quadratic form to zero, we have $q(x+y)=q(x)+q(y)$, so
it follows that $q$ is $\mathbb{Z}$-linear. Moreover, $q(x)=q(-x)=-q(x)$
forces that $2q(x)=0$ for all $x$, but this is $q(2x)$. Hence, $2G$ lies in
the kernel and by the universal property of cokernels, we get a unique map
$f:G/2G\rightarrow M$, giving exactness in the middle. The inclusion
$H_{sym}^{3}(G,M)\subseteq H_{ab}^{3}(G,M)$ follows from the fact that while
symmetric $3$-cocycles have the additional constraint%
\begin{equation}
c(x,y)=-c(y,x)\text{,} \label{lefmu1b}%
\end{equation}
the sets of symmetric vs. abelian $3$-coboundaries are the same (Definition
\ref{def_SymmetricCocycles}). The lower square commutes because the upward
arrows are the same map, just restricted to a subgroup, and on all of
$H_{ab}^{3}(G,M)$ its values are quadratic forms by Lemma
\ref{lemma_TraceIsQuadratic}. Further, Equation \ref{lefmu1b} implies
$2c(x,x)=0$, so the image of the subgroup $H_{sym}^{3}(G,M)$ lands in those
quadratic forms such that the polarization is zero, so it factors over
$\operatorname*{Hom}(G/2G,M)$ by the exactness of the middle row. The upward
arrows being isomorphisms follows from Theorem \ref{thm_EMDeepIso}.
\end{proof}

\begin{lemma}
\label{lemma_FormsOnFreeAbelianGroupArePolar}Suppose $G$ is a free abelian
group, $M$ arbitrary. Then every quadratic form $q:G\rightarrow M$ is polar.
\end{lemma}

The following argument is vaguely analogous to \cite[Lemma D.1]{MR2609644}.

\begin{proof}
Pick a basis $(e_{i})_{i\in I}$ (for some index set $I$) of $G$ as a free
$\mathbb{Z}$-module. Write $b(x,y)$ for the polarization of $q$, as in
Equation \ref{lefmu2b}. Pick an arbitrary total order on the set $I$. Define%
\[
t(e_{i},e_{j}):=\left\{
\begin{array}
[c]{ll}%
b(e_{i},e_{j}) & \text{if }i<j\\
q(e_{i}) & \text{if }i=j\\
0 & \text{if }i>j
\end{array}
\right.
\]
Then if $i\neq j$ we have $t(e_{i},e_{j})+t(e_{j},e_{i})=b(e_{i},e_{j})$ and
for $i=j$ we get $t(e_{i},e_{i})+t(e_{i},e_{i})=2q(e_{i})$, but we also have
$b(e_{i},e_{i})=q(2e_{i})-2q(e_{i})=2q(e_{i})$ since $q(2x)=4q(x)$ for all
$x\in G$ by the same computation as in Equation \ref{lmefa}. Thus, by
$\mathbb{Z}$-linear extension and since $(e_{i})_{i\in I}$ is a basis, we
conclude that $t(x,y)+t(y,x)=b(x,y)$ holds for all $x,y\in G$.
\end{proof}

For an abelian group $M$ we write $\left.  _{n}M\right.  =\ker(M\overset{\cdot
n}{\longrightarrow}M)$.

\begin{lemma}
\label{lemma_CharPolar}Let $t:G\times G\rightarrow M$ be a bilinear map. A
quadratic form $q:G\rightarrow M$

\begin{enumerate}
\item is polar for this $t$, if and only if

\item there exists some $\overline{q}\in\operatorname*{Hom}_{\mathbb{F}_{2}%
}(G/2G,\left.  _{2}M\right.  )$ such that $q(x)=t(x,x)+\overline{q}(x)$ for
all $x\in G$.
\end{enumerate}
\end{lemma}

\begin{proof}
Suppose $q$ is polar for the given $t$. As usual, let
$b(x,y):=q(x+y)-q(x)-q(y)$ be its polarization form, and by assumption we have
$b(x,y)=t(x,y)+t(y,x)$. Define $\overline{q}(x):=q(x)-t(x,x)$ for $x\in G$. We
claim that $\overline{q}$ is a quadratic form. Indeed,%
\[
\overline{q}(-x)=q(-x)-t(-x,-x)=q(x)-t(x,x)=\overline{q}(x)
\]
since $q$ is quadratic and $t$ bilinear. The polarization of $\overline{q}$ is%
\begin{align*}
\overline{b}(x,y)  &  =\overline{q}(x+y)-\overline{q}(x)-\overline{q}(y)\\
&  =q(x+y)-t(x+y,x+y)-q(x)+t(x,x)-q(y)+t(y,y)\\
&  =b(x,y)-(t(x,y)+t(y,x))=0\text{.}%
\end{align*}
We see that this polarization is (trivially) $\mathbb{Z}$-bilinear, completing
the verification that $\overline{q}$ is quadratic, but also proving that
$\overline{q}\in\operatorname*{Hom}_{\mathbb{F}_{2}}(G/2G,\left.
_{2}M\right.  )$ by the exactness of middle row in Lemma
\ref{lemma_WhiteheadExactness}. Conversely, suppose we know that $\overline
{q}$ exists. Then $q$ being polar for $t$ amounts to combining Example
\ref{example_Polar1} and Example \ref{example_Polar2}.
\end{proof}

\subsection{Proof of the main results}

\begin{theorem}
\label{thm_Main}Let $G,M$ be abelian groups. Let $(\beta_{i})_{i\in I}$ be a
basis of the $\mathbb{F}_{2}$-vector space $G/2G$. Let $q\in
\operatorname*{Quad}(G,M)$ be a polar quadratic form. Then it can be written
as%
\[
q(x)=t(x,x)+\overline{q}(x)
\]
(by Lemma \ref{lemma_CharPolar}) for $t:G\times G\rightarrow M$ bilinear and
$\overline{q}\in\operatorname*{Hom}_{\mathbb{F}_{2}}(G/2G,\left.
_{2}M\right.  )$. Define an abelian $3$-cocycle%
\begin{align*}
h(x,y,z)  &  :=0\\
c(x,y)  &  :=t(x,y)+\sum_{i\in I}\overline{x}_{i}\cdot\overline{y}_{i}%
\cdot\overline{q}(\beta_{i})\text{,}%
\end{align*}
where $\overline{x},\overline{y}$ are the vectors we get under the quotient
map $G\twoheadrightarrow G/2G$, spelled out with respect to the basis
$(\beta_{i})$. This means that $\overline{x}_{i},\overline{y}_{i}\in
\mathbb{F}_{2}$. Then the trace of $(h,c)$ is just the quadratic form $q$.
That is, we have found a preimage under%
\begin{equation}
H_{ab}^{3}(G,M)\overset{\sim}{\longrightarrow}\operatorname*{Quad}%
(G,M)\text{.} \label{lefmu4}%
\end{equation}
This preimage is independent of the choice of the basis $(\beta_{i})$.
\end{theorem}

\begin{proof}
\textit{(Step 1) }Note that we can write the polarization of $q$ as%
\begin{equation}
b(x,y)=q(x+y)-q(x)-q(y)=t(x,y)+t(y,x) \label{lefmu6a}%
\end{equation}
for the bilinear form $t:G\times G\rightarrow M$ in the statement of the
theorem by Lemma \ref{lemma_CharPolar}. Next, we check that $(h,c)$ is an
abelian $3$-cocycle. Indeed, having $h(x,y,z)=0$, Identity \ref{lx_a_1} and
\ref{lx_a_2} simplify to the condition that $c:G\times G\rightarrow M$ is
supposed to be $\mathbb{Z}$-bilinear, so this is all we have to show. Since
$t$ is $\mathbb{Z}$-bilinear, it suffices to prove that%
\[
c^{\prime}(x,y):=\sum_{i\in I}\overline{x}_{i}\cdot\overline{y}_{i}%
\cdot\overline{q}(\beta_{i})
\]
is $\mathbb{Z}$-bilinear. However, this is clear since each $\overline
{q}(\beta_{i})\in\left.  _{2}M\right.  $ is a $2$-torsion element; it factors
over $G/2G\times G/2G\rightarrow\left.  _{2}M\right.  $, where it is a
bilinear form on $\mathbb{F}_{2}$-vector spaces.\newline\textit{(Step 2)} Now
that we know that $(h,c)$ is an abelian $3$-cocycle, we need to check that its
trace is $q$. Let us denote its trace by $Q$, i.e. $Q(x):=c(x,x)$. We obtain
the explicit formula%
\[
Q(x)=t(x,x)+\sum_{i\in I}\overline{x}_{i}^{2}\cdot\overline{q}(\beta
_{i})=t(x,x)+\sum_{i\in I}\overline{x}_{i}\cdot\overline{q}(\beta_{i})\text{,}%
\]
where we exploit that the values of $\overline{q}$ lie in the $\mathbb{F}_{2}%
$-vector space $\left.  _{2}M\right.  $, and $n^{2}\equiv n$ in $\mathbb{F}%
_{2}$ for all integers. The polarization of $Q$ is%
\[
B(x,y)=Q(x+y)-Q(x)-Q(y)=t(x,y)+t(y,x)=b(x,y)\text{.}%
\]
However, this is also the polarization of $q$, see Equation \ref{lefmu6a}.
Thus, by the exactness of middle row in Lemma \ref{lemma_WhiteheadExactness}
we deduce that the quadratic form $Q-q$ comes from $\operatorname*{Hom}%
_{\mathbb{F}_{2}}(G/2G,\left.  _{2}M\right.  )$. In particular, $Q-q$ is a
linear map on the $\mathbb{F}_{2}$-vector space $G/2G$. In order to show that
it is zero, it suffices to verify that it is zero on the basis vectors
$\beta_{j}$ with $j\in I$. We compute%
\[
(Q-q)(x)=t(x,x)-q(x)+\sum_{i\in I}\overline{x}_{i}\cdot\left(  \left.
q(\beta_{i})-t(\beta_{i},\beta_{i})\right.  \right)  \text{.}%
\]
So for $x:=\beta_{j}$ we have $\overline{x}_{i}=\delta_{i=j}$ (Kronecker
delta) and thus%
\[
(Q-q)(\beta_{j})=t(\beta_{j},\beta_{j})-q(\beta_{j})+q(\beta_{j})-t(\beta
_{j},\beta_{j})=0\text{.}%
\]
As this also vanishes, Lemma \ref{lemma_WhiteheadExactness} implies that $Q-q$
is the zero quadratic form, i.e. $Q=q$. Finally, our class $(h,c)$ is
independent of the choice of the basis $(\beta_{i})_{i\in I}$ since by the
Eilenberg--Mac\thinspace Lane theorem the trace map in Equation \ref{lefmu4}
is bijective.
\end{proof}

\begin{theorem}
\label{thm_StrictlyAssocSkeletalization}A braided categorical group
$(\mathsf{C},\otimes)$ is braided monoidal equivalent to one which is
simultaneously strictly associative and skeletal if and only if its quadratic
form is polar.
\end{theorem}

\begin{proof}
We find a braided monoidal equivalence $(\mathsf{C},\otimes)\simeq
\mathcal{T}(G,M,(h,c))$ for $G:=\pi_{0}\mathsf{C}$, $M:=\pi_{1}\mathsf{C}$,
and $(h,c)\in H_{ab}^{3}(G,M)$ by the properties of the functor $T$ of Theorem
\ref{thm_JSClassif}. Being skeletal, the triviality of the associator means
that $h(x,y,z)=0$ for all $x,y,z$, Remark
\ref{rmk_StrictlyAssociativeAndSkeletal}. Using this, Identity \ref{lx_a_1}
and \ref{lx_a_2} simplify to say that $c:G\times G\rightarrow M$ is bilinear.
Let $q(x):=c(x,x)$ be its trace, which is quadratic by Lemma
\ref{lemma_TraceIsQuadratic}. For the polarization of $q$ we compute%
\[
b(x,y)=c(x+y,x+y)-c(x,x)-c(y,y)=c(x,y)+c(y,x)\text{.}%
\]
Thus, $q$ is polar, because we may use $t(x,y):=c(x,y)$ in Definition
\ref{def_PolarQuadraticMap}. Conversely, suppose $(h,c)$ under the
Eilenberg--Mac\thinspace Lane isomorphism of Equation \ref{lEMMa} gets sent to
a polar quadratic form. Then we may apply Theorem \ref{thm_Main} and it
produces a cohomologous abelian $3$-cocycle representative of the shape
$(0,c^{\prime})$, so reversely by Remark
\ref{rmk_StrictlyAssociativeAndSkeletal} we get a braided monoidal equivalence%
\[
(\mathsf{C},\otimes)\simeq\mathcal{T}(G,M,(0,c^{\prime}))\text{,}%
\]
from Theorem \ref{thm_JSClassif}, but the right side is both skeletal and
strictly associative.
\end{proof}

\begin{theorem}
\label{thm_PolarCover}For every braided categorical group $(\mathsf{C}%
,\otimes)$, there exists an essentially surjective and faithful (but typically
not full) braided monoidal functor from another braided categorical group%
\begin{equation}
(\widehat{\mathsf{C}},\otimes)\longrightarrow(\mathsf{C},\otimes)\text{,}
\label{l_1H}%
\end{equation}
surjective on $\pi_{0}$, and an isomorphism on $\pi_{1}$, such that
$(\widehat{\mathsf{C}},\otimes)$ is simultaneously strictly associative and
skeletal. We call $(\widehat{\mathsf{C}},\otimes)$ a \emph{polar cover} of
$(\mathsf{C},\otimes)$.
\end{theorem}

\begin{proof}
By the Joyal--Street classification, there exists an equivalence%
\[
F:(\mathsf{C},\otimes)\overset{\sim}{\rightarrow}\mathcal{T}(\pi_{0}%
\mathsf{C},\pi_{1}\mathsf{C},(h,c))
\]
for $(h,c)$ an abelian $3$-cocycle representing the relevant cohomology class.
Let $c:P\twoheadrightarrow\pi_{0}\mathsf{C}$ be a surjective homomorphism from
a free abelian group $P$ (e.g., pick a projective cover of $\pi_{0}\mathsf{C}$
as a $\mathbb{Z}$-module). We get a commutative square%
\[%
\xymatrix{
P \ar@{->>}[r]^{c} \ar[d]_{q \circ c} & {  \pi_{0}{\mathsf{C}} } \ar[d]^{q} \\
\pi_{1}{\mathsf{C}} \ar@{=}[r] & \pi_{1}{\mathsf{C}},
}%
\]
describing a morphism $(P,\pi_{1}\mathsf{C},q\circ c)\rightarrow(\pi
_{0}\mathsf{C},\pi_{1}\mathsf{C},q)$ in the category $\mathcal{Q}uad$ of
\S \ref{sect_JoyalStreetClassif}. Note that $q\circ c$ is indeed again a
quadratic form since $c$ is $\mathbb{Z}$-linear. As the functor $T$ of Theorem
\ref{thm_JSClassif} is full and essentially surjective, and the skeletal
categorical groups $\mathcal{T}(-,-,-)$ are concrete representatives for the
essential surjectivity, we obtain some braided monoidal functor%
\[
G:\mathcal{T}(P,\pi_{1}\mathsf{C},(h^{\prime},c^{\prime}))\longrightarrow
\mathcal{T}(\pi_{0}\mathsf{C},\pi_{1}\mathsf{C},(h,c))\text{,}%
\]
where $(h^{\prime},c^{\prime})$ is an abelian $3$-cocycle corresponding to the
lifted quadratic form $q\circ c$ under the\ Eilenberg--Mac\thinspace Lane
isomorphism of Equation \ref{lEMMa}. Next, by Lemma
\ref{lemma_FormsOnFreeAbelianGroupArePolar} since $P$ is free abelian, the
quadratic form $q\circ c$ is necessarily polar and thus has a representative
of the shape $(0,c^{\prime\prime})$ within its cohomology class in $H_{ab}%
^{3}(P,\pi_{1}\mathsf{C})$ by Theorem \ref{thm_Main}. Hence, we can
pre-compose $G$ with a braided monoidal equivalence%
\[
\mathcal{T}(P,\pi_{1}\mathsf{C},(0,c^{\prime\prime}))\overset{\sim
}{\longrightarrow}\mathcal{T}(P,\pi_{1}\mathsf{C},(h^{\prime},c^{\prime
}))\longrightarrow\mathcal{T}(\pi_{0}\mathsf{C},\pi_{1}\mathsf{C},(h,c))
\]
and as explained in\ Remark \ref{rmk_StrictlyAssociativeAndSkeletal} the
braided categorical group $\mathcal{T}(P,\pi_{1}\mathsf{C},(0,c^{\prime\prime
}))$ is both skeletal and strictly associative. Call it $(\widehat{\mathsf{C}%
},\otimes)$. Using the above composition of functors, we get the braided
monoidal functor of\ Equation \ref{l_1H}; the surjection on the level of
$\pi_{0}$ is the map $c$, and the isomorphy on the level of $\pi_{1}$ is
clear. As the map on the level of $\pi_{0}$ is surjective, the functor is
essentially surjective. On the level of morphisms only two things can happen:
Morphisms are $\pi_{1}\mathsf{C}$ if the objects are isomorphic, or vacuous
otherwise. It follows that the functor is faithful, but not necessarily full
since non-isomorphic objects may become isomorphic, namely when two distinct
elements of $P$ map to the same element in $\pi_{0}\mathsf{C}$.
\end{proof}

%

\appendix

\section{Strictifying the universal determinant}

Finally, we explain a consequence of our results to Deligne's universal
determinant functor of an exact category. Allowing ourselves an anachronistic
interpretation, the idea of the universal determinant functor is as follows:

Take the truncation of the $K$-theory spectrum $K(\mathsf{C})$ of an exact
category $\mathsf{C}$ to its stable homotopy $[0,1]$-type. We get a map
$K(\mathsf{C})\rightarrow\tau_{\leq1}K(\mathsf{C})$ and using a stable variant
of Equation \ref{lzx1} (e.g., concretely \cite{MR2981952} or \cite[\S 5]%
{MR2981817}), the stable homotopy $[0,1]$-type $\tau_{\leq1}K(\mathsf{C})$ can
be modelled as a Picard groupoid. Deligne has observed in his paper
\cite[\S 4.2]{MR902592} that this map and the relevant Picard groupoid can
equivalently be described as the target of the universal determinant functor
defined on the category $\mathsf{C}$, giving a formulation which is a priori
independent of any algebraic $K$-theory. This yields a connection to the far
less homotopically defined determinant functors in terms of top exterior
powers of vector bundles, as they would occur in algebraic geometry (e.g., the
determinant line bundle on moduli spaces of vector bundles on curves, as just
one possible application).

Now, being a Picard groupoid, it already follows from the result of
Johnson--Osorno \cite{MR2981952} that $\tau_{\leq1}K(\mathsf{C})$ can be made
skeletal and strictly associative. In this appendix we explain how to pin down
the relevant symmetry constraint, using the formula from our Theorem
\ref{thm_Main}. It will also follow readily that besides the target,
\textit{the entire universal determinant functor can be strictified}.

We have not found this fact recorded anywhere in the literature.\medskip

In detail: Suppose $\mathsf{C}$ is an exact category with a fixed zero object.
Deligne constructs the Picard groupoid of virtual objects $V(\mathsf{C})$: Let
$Q\mathsf{C}$ be the Quillen $Q$-construction of $\mathsf{C}$; see for example
\cite[Chapter IV, \S 6]{MR3076731} for a precise definition. Write
$N_{\bullet}Q\mathsf{C}$ for its nerve. As $Q\mathsf{C}$ has the same objects
as $\mathsf{C}$ itself, the fixed zero object pins down a $0$-simplex of
$N_{\bullet}Q\mathsf{C}$, rendering the latter a pointed simplicial set.

Set up a new category $V(\mathsf{C})$

\begin{enumerate}
\item whose objects are closed loops in the space $N_{\bullet}Q\mathsf{C}$
around the base point, and

\item whose morphisms are based homotopy classes of homotopies between loops.
\end{enumerate}

The composition of morphisms is defined as the composition of homotopies. The
associativity law for composition then holds (and only holds) because
morphisms are only taken modulo their based homotopy class.

One checks that the above makes $V(\mathsf{C})$ a groupoid. A monoidal
structure on $V(\mathsf{C})$, i.e. a suitable bifunctor $\otimes
:V(\mathsf{C})\times V(\mathsf{C})\rightarrow V(\mathsf{C})$ is defined as the
composition of loops on the level of objects. This bifunctor can be promoted
to a symmetric monoidal structure. We refer to \cite[\S 4.2]{MR902592} for
further details.

In \cite[\S 4.3]{MR902592} Deligne gives a second construction of
$V(\mathsf{C})$. He first sets up the concept of a determinant functor. Given
any exact category $\mathsf{C}$, we write $\mathsf{C}^{\times}$ for the same
category, except that we only keep isomorphisms as morphisms (this is
sometimes called the \emph{maximal inner groupoid} or \emph{group core}). This
is a groupoid.

\begin{definition}
[{\cite[\S 4.3]{MR902592}}]\label{def_DeterminantFunctor}Let $\mathsf{C}$ be
an exact category and let $(\mathsf{P},\otimes)$ be a Picard groupoid. A
\emph{determinant functor} on $\mathsf{C}$ is a functor $\mathcal{D}%
:\mathsf{C}^{\times}\rightarrow\mathsf{P}$ along with the following extra
structure and axioms:

\begin{enumerate}
\item For any exact sequence $\Sigma:G^{\prime}\hookrightarrow
G\twoheadrightarrow G^{\prime\prime}$ in $\mathsf{C}$, we are given an
isomorphism%
\[
\mathcal{D}(\Sigma):\mathcal{D}(G)\overset{\sim}{\longrightarrow}%
\mathcal{D}(G^{\prime})\underset{\mathsf{P}}{\otimes}\mathcal{D}%
(G^{\prime\prime})
\]
in $\mathsf{P}$. This isomorphism is required to be functorial in morphisms of
exact sequences.

\item For every zero object $Z$ of $\mathsf{C}$, we are given an isomorphism
$z:\mathcal{D}(Z)\overset{\sim}{\rightarrow}1_{\mathsf{P}}$ to the neutral
object of the Picard groupoid.

\item Suppose $f:G\rightarrow G^{\prime}$ is an isomorphism in $\mathsf{C}$.
We write%
\[
\Sigma_{l}:0\hookrightarrow G\twoheadrightarrow G^{\prime}\qquad
\text{and}\qquad\Sigma_{r}:G\hookrightarrow G^{\prime}\twoheadrightarrow0
\]
for the depicted exact sequences. We demand that the composition%
\begin{equation}
\mathcal{D}(G)\underset{\mathcal{D}(\Sigma_{l})}{\overset{\sim}%
{\longrightarrow}}\mathcal{D}(0)\underset{\mathsf{P}}{\otimes}\mathcal{D}%
(G^{\prime})\underset{z\otimes1}{\overset{\sim}{\longrightarrow}}%
1_{\mathsf{P}}\underset{\mathsf{P}}{\otimes}\mathcal{D}(G^{\prime}%
)\underset{\mathsf{P}}{\overset{\sim}{\longrightarrow}}\mathcal{D}(G^{\prime})
\label{l_CDetFunc1}%
\end{equation}
and the natural map $\mathcal{D}(f):\mathcal{D}(G)\overset{\sim}{\rightarrow
}\mathcal{D}(G^{\prime})$ agree. We further require that $\mathcal{D}(f^{-1})$
agrees with a variant of Equation \ref{l_CDetFunc1} using $\Sigma_{r}$ instead
of $\Sigma_{l}$.

\item If a two-step filtration $G_{1}\hookrightarrow G_{2}\hookrightarrow
G_{3}$ is given, we demand that the diagram%
\begin{equation}%
\xymatrix{
\mathcal{D}(G_3) \ar[r]^-{\sim} \ar[d]_{\sim} & \mathcal{D}(G_1) \otimes
\mathcal{D}(G_3/G_1) \ar[d]^{\sim} \\
\mathcal{D}(G_2) \otimes\mathcal{D}(G_3/G_2) \ar[r]_-{\sim} & \mathcal
{D}(G_1) \otimes\mathcal{D}(G_2/G_1) \otimes\mathcal{D}(G_3/G_2)
}
\label{l2}%
\end{equation}
commutes.

\item Given objects $G,G^{\prime}\in\mathsf{C}$ consider the exact sequences%
\[
\Sigma_{1}:G\hookrightarrow G\oplus G^{\prime}\twoheadrightarrow G^{\prime
}\qquad\text{and}\qquad\Sigma_{2}:G^{\prime}\hookrightarrow G\oplus G^{\prime
}\twoheadrightarrow G
\]
with the natural inclusion and projection morphisms. Then the diagram%
\begin{equation}%
\xymatrix{
& \mathcal{D}(G \oplus G^{\prime}) \ar[dl]_{\mathcal{D}(\Sigma_1)}
\ar[dr]^{\mathcal{D}(\Sigma_2)} \\
\mathcal{D}(G) \otimes\mathcal{D}(G^{\prime}) \ar[rr]_{s_{G,G^{\prime}}}
& & \mathcal{D}(G^{\prime}) \otimes\mathcal{D}(G)
}
\label{lz1a}%
\end{equation}
commutes, where $s_{G,G^{\prime}}$ denotes the symmetry constraint of
$\mathsf{P}$.
\end{enumerate}
\end{definition}

At the end of \cite[\S 4.3]{MR902592} Deligne now considers the category
$\det(\mathsf{C},\mathsf{P})$ of determinant functors, i.e.

\begin{enumerate}
\item objects are determinant functors in the sense of the above definition, and

\item morphisms are natural transformations of determinant functors.
\end{enumerate}

Details for this are spelled out in \cite[\S 2.3]{MR2842932}, especially a
full description of a morphism of determinant functors is \cite[Definition
2.5]{MR2842932}. We also took over his notation $\det(\mathsf{C},\mathsf{P})$
for this category.

\begin{definition}
\label{def_UnivDetFunctor}A determinant functor $\mathcal{D}:\mathsf{C}%
^{\times}\longrightarrow\mathsf{P}$ is called \emph{universal} if for every
given Picard groupoid $\mathsf{P}^{\prime}$ the functor%
\[
\operatorname*{Hom}\nolimits^{\otimes}(\mathsf{P},\mathsf{P}^{\prime
})\longrightarrow\det\left(  \mathsf{C},\mathsf{P}^{\prime}\right)
\,\text{,}\qquad\varphi\mapsto\varphi\circ\mathcal{D}%
\]
is an equivalence of categories.
\end{definition}

This is already in Deligne \cite[\S 4.3]{MR902592}, but perhaps a little more
detailed in \cite[\S 4.1]{MR2842932}. Deligne then argues that a universal
determinant functor exists and can be constructed using $V(\mathsf{C})$. To
set it up, recall that the $Q$-construction category $Q\mathsf{C}$ has the
same objects as $\mathsf{C}$, and for every admissible monic (resp. epic) $f$
in $\mathsf{C}$, there are arrows $f_{!}$ (resp. $f^{!}$) in $Q\mathsf{C}$;
see for example \cite[\S 2]{MR0338129} or \cite[Chapter 6]{MR1382659}. Let $0$
be the fixed zero object of $\mathsf{C}$. We use the notation%
\[
0_{!}^{A}=(0\twoheadleftarrow0\hookrightarrow A)\qquad\text{and}\qquad
0_{A}^{!}=(0\twoheadleftarrow A\hookrightarrow A)\text{,}%
\]
using the canonical arrows coming from the fact that $0$ is both initial and
final in $\mathsf{C}$. For every object $X\in\mathsf{C}$ one considers the
diagram%
\begin{equation}%
\xymatrix{
0 \ar[rr]^{0_{!}^{X}}  & & X & & 0 \ar[ll]_{0_{X}^{!}}
}
\label{lal1}%
\end{equation}
showing that $(0_{!}^{A})^{-1}\circ0_{A}^{!}$ is a closed loop around the base
point in the nerve of $Q\mathsf{C}$. We denote it by $[X]$. We can now
formulate Deligne's fundamental result. Recall that defining a morphism in
$V(\mathsf{C})$ can be done by pinning down a homotopy of loops.

\begin{theorem}
[{Deligne, \cite[\S 4.4-4.5]{MR902592}}]Let $\mathsf{C}$ be an exact category
with a fixed zero object $0$. Then $V(\mathsf{C})$ is a Picard groupoid.
Define a functor $\mathcal{D}:\mathsf{C}^{\times}\rightarrow V(\mathsf{C})$ by%
\[
X\mapsto\lbrack X]
\]
on objects, i.e. the loop of Equation \ref{lal1}. Isomorphisms $\varphi
:X\rightarrow Y$ give rise to a homotopy of loops%
\[%
\xymatrix@!=0.015in{
& & X \ar[dd]^{\varphi} \\
0 \ar[urr]^{0_{!}^{X}} \ar[drr]_{0_{!}^{Y}}  & & & & 0 \ar[ull]_{0_{X}^{!}}
\ar[dll]^{0_{Y}^{!}} \\
& & Y
}%
\]
and this defines the functor on morphisms. To any exact sequence%
\[
\left.  \Sigma:\right.  \left.  A\overset{\alpha}{\hookrightarrow}%
B\overset{\beta}{\twoheadrightarrow}C\right.
\]
attach the homotopy $\mathcal{D}(\Sigma)$
\begin{equation}%
\xymatrix@!=0.015in{
0 \ar[rr]^{0_{!}^{B}} \ar[dr]_{0_{!}^{A}} & & B & & 0 \ar[dl]^{0_{C}^{!}}
\ar[ll]_{0_{B}^{!}}   \\
& A \ar[ur]_{\alpha_{!}} & & C \ar[ul]^{\beta^{!}} \\
& & 0, \ar@{-->}[uu] \ar[ul]^{0^{!}_{A}} \ar[ur]_{0_{!}^{C}}
}
\label{l1}%
\end{equation}
which are four $2$-simplices giving a homotopy between the required loops. The
dashed arrow is $(0\twoheadleftarrow A\overset{\alpha}{\hookrightarrow}B)$ in
$Q\mathsf{C}$. Then $\mathcal{D}$ is a universal determinant functor.
\end{theorem}

Usually, although Definition \ref{def_UnivDetFunctor} is a little more subtle
in its formulation, this is (very reasonably) simply called \textit{the}
universal determinant functor. Next, we shall apply the strictification
methods of \S \ref{sect_Polar} to Deligne's constructions. First of all,
$V(\mathsf{C})$ is a Picard groupoid, i.e. a braided categorical group whose
braiding is symmetric. We apply Joyal--Street's skeletalization, giving a
braided (symmetric) monoidal equivalence%
\begin{equation}
V(\mathsf{C})\overset{\sim}{\longrightarrow}\mathcal{T}(\pi_{0}V(\mathsf{C}%
),\pi_{1}V(\mathsf{C}),(h,c))\text{,} \label{lam}%
\end{equation}
where $[(h,c)]$ is the abelian $3$-cohomology class in $H_{ab}^{3}(\pi
_{0}V(\mathsf{C}),\pi_{1}V(\mathsf{C}))$ encoding the braiding and
associativity constraint. Since we have canonical isomorphisms%
\[
\pi_{0}V(\mathsf{C})\cong K_{0}(\mathsf{C})\qquad\text{and}\qquad\pi
_{1}V(\mathsf{C})\cong K_{1}(\mathsf{C})\text{,}%
\]
using that the virtual objects arise as the truncation of the $K$-theory
spectrum to its stable $[0,1]$-type, we understand these groups. Moreover,
since the braiding is symmetric, this cocycle lies in the subgroup
$H_{sym}^{3}\subseteq H_{ab}^{3}$, as in Lemma \ref{lemma_WhiteheadExactness}.

It follows from the exact middle row in Lemma \ref{lemma_WhiteheadExactness}
that the polarization form of the quadratic form $q\in\operatorname*{Quad}%
(K_{0}(\mathsf{C}),K_{1}(\mathsf{C}))$ is zero. In particular, $q$ is polar;
just pick $t(x,y):=0$ in Definition \ref{def_PolarQuadraticMap}. Now apply
Theorem \ref{thm_Main}. Let $(\beta_{i})_{i\in I}$ be a basis of the
$\mathbb{F}_{2}$-vector space $K_{0}(\mathsf{C})/2K_{0}(\mathsf{C})$. We have
$\overline{q}=q$ (in the notation of the cited theorem) since $t$ vanishes, so
the symmetric $3$-cocycle%
\[
h(x,y,z):=0\qquad\qquad c(x,y):=\sum_{i\in I}\overline{x}_{i}\cdot\overline
{y}_{i}\cdot q(\beta_{i})
\]
is a representative of the symmetric cohomology class of $[(h,c)]$. Without
loss of generality, we may assume that this is the representative we had
started with. We next compute the $q(\beta_{i})$ in terms of the original
exact category $\mathsf{C}$. Returning to Lemma \ref{lemma_WhiteheadExactness}%
, we see that $q$ is an $\mathbb{F}_{2}$-linear map $K_{0}(\mathsf{C}%
)/2K_{0}(\mathsf{C})\rightarrow\left.  _{2}K_{1}(\mathsf{C})\right.  $. As
$K_{0}(\mathsf{C})$ is the algebraic group completion of the monoid of
isomorphism classes of objects in $\mathsf{C}$, all its elements have the
shape $[X]-[Y]$ with $X,Y\in\mathsf{C}$ objects. In particular, for each
$\beta_{i}$ pick%
\begin{equation}
\beta_{i}=\overline{[X]-[Y]}, \label{l1d}%
\end{equation}
where $X,Y\in\mathsf{C}$ are objects. We have $q(\beta_{i})=q([X])-q([Y])$
since $q$ is linear, and $q(x)=c(x,x)$, meaning that%
\[
q([X])=s_{X,X}:\mathcal{D}(X)\otimes\mathcal{D}(X)\longrightarrow
\mathcal{D}(X)\otimes\mathcal{D}(X)\text{,}%
\]
referring to the symmetry constraint of the virtual objects $V(\mathsf{C})$.
In order to compute this map $s_{X,X}$, we can rely on%
\[%
\xymatrix{
& \mathcal{D}(X \oplus X) \ar[dl]_{\mathcal{D}(\Sigma_1)}
\ar[dr]^{\mathcal{D}(\Sigma_2)} \\
\mathcal{D}(X) \otimes\mathcal{D}(X) \ar[rr]_{s_{X,X}}
& & \mathcal{D}(X) \otimes\mathcal{D}(X)
}%
\]
(which is Diagram \ref{lz1a} in the special case of using the same object).
The bottom arrow is the one we need, but the upper arrows come from
$\mathcal{D}(\Sigma_{1})$ resp. $\mathcal{D}(\Sigma_{2})$, still in the
notation loc. cit., which only differ by the swapping isomorphism in the
middle term of $X\hookrightarrow X\oplus X\twoheadrightarrow X$, and the
identity map on the two outer copies of $X$. In the attached homotopies, as in
Equation \ref{l1}, also all vertices agree, and all outer edges agree, and
only the three inner edges are different, because they differ exactly by the
swapping of the two copies of $X$.\ For any object $X\in\mathsf{C}$, the
swapping map%
\[
X\oplus X\longrightarrow X\oplus X\text{,}\qquad x_{1}\oplus x_{2}\mapsto
x_{2}\oplus x_{1}%
\]
is an automorphism, and thus functorially induces an automorphism of
$\mathcal{D}(X\oplus X)$, meaning a homotopy, and this homotopy agrees with
the one of $\mathcal{D}(\Sigma_{2})^{-1}\circ\mathcal{D}(\Sigma_{1})$. It
follows that $s_{X,X}\in K_{1}(\mathsf{C})$ is just the automorphism attached
to the swapping map of $X\oplus X$.

This finishes describing $\mathcal{T}(\pi_{0}V(\mathsf{C}),\pi_{1}%
V(\mathsf{C}),(0,c))$ in detail. Next, we make the composition%
\[
\det\nolimits^{\operatorname*{strict}}:\mathsf{C}^{\times}\longrightarrow
V(\mathsf{C})\overset{\sim}{\longrightarrow}\mathcal{T}(\pi_{0}V(\mathsf{C}%
),\pi_{1}V(\mathsf{C}),(h,c))
\]
a determinant functor. It is already a functor, and is taking values in a
skeletal and strictly associative Picard groupoid. Explicitly:

\begin{observe}
The functor $\det\nolimits^{\operatorname*{strict}}$ has the following properties:

\begin{enumerate}
\item on the level of objects, it sends $X\in\mathsf{C}$ to its $K_{0}$-class
$[X]\in K_{0}(\mathsf{C})$,

\item on the level of automorphisms $\varphi:X\rightarrow X$, it sends
$\varphi$ to its $K_{1}$-class $[X,\varphi]\in K_{1}(\mathsf{C})$.
\end{enumerate}
\end{observe}

Here we use that $K_{0}(\mathsf{C})$ has the well-known explicit description
of being the free abelian group on isomorphism classes of objects in
$\mathsf{C}$ (and quotient out $[X]=[X^{\prime}]+[X^{\prime\prime}]$ whenever
an exact sequence $X^{\prime}\hookrightarrow X\twoheadrightarrow
X^{\prime\prime}$ exists); and that $K_{1}(\mathsf{C})$ contains canonical
elements attached to automorphisms of objects (see for example \cite[Chapter
IV, Exercise 8.7, or Example 9.6.2]{MR3076731}). If $\mathsf{C}=P_{f}(R)$ is
the category of finitely generated projective modules over a ring $R$, the
group $K_{1}$ is actually generated by the classes coming from such
automorphisms (this can also be used to give a presentation of the $K_{1}%
$-group; there is a discussion of this in \cite{MR1637539}).

We do not get a good description of morphisms on all of $\mathsf{C}^{\times}$,
since isomorphisms between different objects will be sent to the class of the
automorphism gotten once having pre- and post-composed the concrete choices of
isomorphisms between the objects of $\mathsf{C}$ and our choices tacitly made
when picking a skeleton earlier, so this will not admit a choice-free
description. Certainly, if one only needs to do a computation on finitely many
objects, all this data can be chosen concretely and unravelled.

To promote $\det\nolimits^{\operatorname*{strict}}:\mathsf{C}^{\times
}\rightarrow\mathcal{T}(\pi_{0}V(\mathsf{C}),\pi_{1}V(\mathsf{C}),(0,c))$ to a
genuine determinant functor, it only remains to transport the datum
$\mathcal{D}(\Sigma)$ attached to exact sequences $\Sigma$. However, given
$\Sigma:G^{\prime}\hookrightarrow G\twoheadrightarrow G^{\prime\prime}$ in
$\mathsf{C}$, then under the equivalence of Equation \ref{lam} in the
isomorphism%
\[
\mathcal{D}(\Sigma):\mathcal{D}(G)\overset{\sim}{\longrightarrow}%
\mathcal{D}(G^{\prime})\underset{\mathsf{P}}{\otimes}\mathcal{D}%
(G^{\prime\prime})
\]
both sides are the same because the Picard groupoid is skeletal. Thus,
$\mathcal{D}(\Sigma)$ really only defines an automorphism of an object after
going all to $\mathcal{T}(\pi_{0}V(\mathsf{C}),\pi_{1}V(\mathsf{C}),(0,c))$.
Such an automorphism is an element $B_{\Sigma}\in K_{1}(\mathsf{C})$. Equation
\ref{l2} now translates to the cocycle type identity%
\[
B_{G_{2}\hookrightarrow G_{3}\twoheadrightarrow G_{3}/G_{2}}-B_{G_{1}%
\hookrightarrow G_{3}\twoheadrightarrow G_{3}/G_{1}}+B_{G_{1}\hookrightarrow
G_{2}\twoheadrightarrow G_{2}/G_{1}}-B_{G_{2}/G_{1}\hookrightarrow G_{3}%
/G_{1}\twoheadrightarrow G_{3}/G_{2}}=0\text{.}%
\]

\begin{observe}
The functor $\det\nolimits^{\operatorname*{strict}}$ is a universal
determinant functor.
\end{observe}

This holds because we constructed it from a universal one by symmetric
monoidal equivalence (Definition \ref{def_UnivDetFunctor}). Summarizing, this
shows that at least abstractly there is a strict model for Deligne's universal
determinant functor. Of course, in concrete terms, picking a genuine skeleton
of $V(\mathsf{C})$ can be hard or impossible, depending on $\mathsf{C}$. For
$\mathsf{C}=\mathsf{Vect}_{f}(k)$ being finite-dimensional vector spaces over
a field, it can be done. Just sketching this, for each $n\in K_{0}%
(k)=\mathbb{Z}$ pick the object $X:=k^{n}$ if $n\geq0$ and $k^{-n}$ for $n<1$,
along with the standard basis. Then for an arbitrary finite-dimensional vector
space picking a basis amounts to making the isomorphism to some $k^{n}$
explicit. The symmetry constraint can be computed using Equation \ref{l1d}.
One gets the standard Koszul type sign of the determinant line by observing
that the matrix of the swapping map goes to $+1$ or $-1$ in $K_{1}$, depending
on the rank (see \cite[Chapter III, Example 1.2.1]{MR3076731}). Finally,
working out the $B_{\Sigma}\in K_{1}(\mathsf{C})$ yields exactly the
well-known rules for the top exterior power of the usual graded determinant
line (restricted to this skeleton).

\bibliographystyle{amsalpha}
\bibliography{ollinewbib}

\def\cprime{$'$} \def\polhk#1{\setbox0=\hbox{#1}{\ooalign{\hidewidth
  \lower1.5ex\hbox{`}\hidewidth\crcr\unhbox0}}} \def\cprime{$'$}
  \def\cprime{$'$} \def\cprime{$'$} \def\cprime{$'$}
\providecommand{\bysame}{\leavevmode\hbox to3em{\hrulefill}\thinspace}
\providecommand{\MR}{\relax\ifhmode\unskip\space\fi MR }
\providecommand{\MRhref}[2]{%
  \href{http://www.ams.org/mathscinet-getitem?mr=#1}{#2}
}
\providecommand{\href}[2]{#2}
\begin{thebibliography}{DGNO10}

\bibitem[Bau91]{MR1096295}
H.~Baues, \emph{Combinatorial homotopy and {$4$}-dimensional complexes}, De
  Gruyter Expositions in Mathematics, vol.~2, Walter de Gruyter \& Co., Berlin,
  1991, With a preface by Ronald Brown. \MR{1096295}

\bibitem[BCC93]{MR1219923}
M.~Bullejos, P.~Carrasco, and A.~M. Cegarra, \emph{Cohomology with coefficients
  in symmetric cat-groups. {A}n extension of {E}ilenberg-{M}ac {L}ane's
  classification theorem}, Math. Proc. Cambridge Philos. Soc. \textbf{114}
  (1993), no.~1, 163--189. \MR{1219923}

\bibitem[BL04]{MR2068521}
J.~Baez and A.~Lauda, \emph{Higher-dimensional algebra. {V}. 2-groups}, Theory
  Appl. Categ. \textbf{12} (2004), 423--491. \MR{2068521}

\bibitem[Bre11]{MR2842932}
M.~Breuning, \emph{Determinant functors on triangulated categories}, J.
  K-Theory \textbf{8} (2011), no.~2, 251--291. \MR{2842932}

\bibitem[CC96]{MR1414569}
P.~Carrasco and A.~M. Cegarra, \emph{({B}raided) tensor structures on homotopy
  groupoids and nerves of (braided) categorical groups}, Comm. Algebra
  \textbf{24} (1996), no.~13, 3995--4058. \MR{1414569}

\bibitem[CK07]{MR2293318}
A.~M. Cegarra and E.~Khmaladze, \emph{Homotopy classification of braided graded
  categorical groups}, J. Pure Appl. Algebra \textbf{209} (2007), no.~2,
  411--437. \MR{2293318}

\bibitem[Del87]{MR902592}
P.~Deligne, \emph{Le d\'eterminant de la cohomologie}, Current trends in
  arithmetical algebraic geometry ({A}rcata, {C}alif., 1985), Contemp. Math.,
  vol.~67, Amer. Math. Soc., Providence, RI, 1987, pp.~93--177. \MR{902592
  (89b:32038)}

\bibitem[DGNO10]{MR2609644}
V.~Drinfeld, S.~Gelaki, D.~Nikshych, and V.~Ostrik, \emph{On braided fusion
  categories. {I}}, Selecta Math. (N.S.) \textbf{16} (2010), no.~1, 1--119.
  \MR{2609644}

\bibitem[Dur77]{MR480333}
A.~Durfee, \emph{Bilinear and quadratic forms on torsion modules}, Advances in
  Math. \textbf{25} (1977), no.~2, 133--164. \MR{480333}

\bibitem[EML53]{MR0056295}
S.~Eilenberg and S.~Mac~Lane, \emph{On the groups {$H(\Pi,n)$}. {I}}, Ann. of
  Math. (2) \textbf{58} (1953), 55--106. \MR{0056295}

\bibitem[GJ16]{MR3605649}
C.~Galindo and N.~Jaramillo, \emph{Solutions of the hexagon equation for
  abelian anyons}, Rev. Colombiana Mat. \textbf{50} (2016), no.~2, 273--294.
  \MR{3605649}

\bibitem[GMdR02]{MR1890928}
A.~R. Garz\'{o}n, J.~G. Miranda, and A.~del R\'{\i}o, \emph{Tensor structures
  on homotopy groupoids of topological spaces}, Int. Math. J. \textbf{2}
  (2002), no.~5, 407--431. \MR{1890928}

\bibitem[Isb69]{MR0249484}
J.~R. Isbell, \emph{On coherent algebras and strict algebras}, J. Algebra
  \textbf{13} (1969), 299--307. \MR{0249484}

\bibitem[JO12]{MR2981952}
N.~Johnson and A.~Osorno, \emph{Modeling stable one-types}, Theory Appl. Categ.
  \textbf{26} (2012), No. 20, 520--537. \MR{2981952}

\bibitem[JS86]{joyalstreetpreprint}
A.~Joyal and R.~Street, \emph{Braided monoidal categories}, Macquarie
  Mathematical Reports (1986), no.~860081.

\bibitem[JS93]{MR1250465}
\bysame, \emph{Braided tensor categories}, Adv. Math. \textbf{102} (1993),
  no.~1, 20--78. \MR{1250465}

\bibitem[Lap83]{MR723395}
M.~Laplaza, \emph{Coherence for categories with group structure: an alternative
  approach}, J. Algebra \textbf{84} (1983), no.~2, 305--323. \MR{723395}

\bibitem[ML52]{MR0045115}
S.~Mac~Lane, \emph{Cohomology theory of {A}belian groups}, Proceedings of the
  {I}nternational {C}ongress of {M}athematicians, {C}ambridge, {M}ass., 1950,
  vol. 2, Amer. Math. Soc., Providence, R. I., 1952, pp.~8--14. \MR{0045115}

\bibitem[Nen98]{MR1637539}
A.~Nenashev, \emph{{$K_1$} by generators and relations}, J. Pure Appl. Algebra
  \textbf{131} (1998), no.~2, 195--212. \MR{1637539}

\bibitem[NSW08]{MR2392026}
J.~Neukirch, A.~Schmidt, and K.~Wingberg, \emph{Cohomology of number fields},
  second ed., Grundlehren der Mathematischen Wissenschaften [Fundamental
  Principles of Mathematical Sciences], vol. 323, Springer-Verlag, Berlin,
  2008. \MR{2392026}

\bibitem[Pat12]{MR2981817}
D.~Patel, \emph{de {R}ham epsilon factors}, Invent. Math. \textbf{190} (2012),
  no.~2, 299--355. \MR{2981817}

\bibitem[Qui73]{MR0338129}
D.~Quillen, \emph{Higher algebraic {$K$}-theory. {I}}, Algebraic {$K$}-theory,
  {I}: {H}igher {$K$}-theories ({P}roc. {C}onf., {B}attelle {M}emorial {I}nst.,
  {S}eattle, {W}ash., 1972), Springer, Berlin, 1973, pp.~85--147. Lecture Notes
  in Math., Vol. 341. \MR{0338129 (49 \#2895)}

\bibitem[Qui99]{MR1734419}
F.~Quinn, \emph{Group categories and their field theories}, Proceedings of the
  {K}irbyfest ({B}erkeley, {CA}, 1998), Geom. Topol. Monogr., vol.~2, Geom.
  Topol. Publ., Coventry, 1999, pp.~407--453. \MR{1734419}

\bibitem[S{\'{\i}}n75]{sinh}
H.~X. S{\'{\i}}nh, \emph{{$Gr$-cat\'egories" (thesis, handwritten
  manuscript)}}, Universit\'e Paris 7,
  https://pnp.mathematik.uni-stuttgart.de/lexmath/kuenzer/sinh.html, 1975.

\bibitem[Sri96]{MR1382659}
V.~Srinivas, \emph{Algebraic {$K$}-theory}, second ed., Progress in
  Mathematics, vol.~90, Birkh\"{a}user Boston, Inc., Boston, MA, 1996.
  \MR{1382659}

\bibitem[Var]{BAEZPOST}
Various, \emph{{{B}raided 2-{G}roups from {L}attices}},
  \url{https://golem.ph.utexas.edu/category/2015/01/integral_octonions_part_12.html}.

\bibitem[Wal63]{MR156890}
C.~T.~C. Wall, \emph{Quadratic forms on finite groups, and related topics},
  Topology \textbf{2} (1963), 281--298. \MR{156890}

\bibitem[Wei94]{MR1269324}
C.~Weibel, \emph{An introduction to homological algebra}, Cambridge Studies in
  Advanced Mathematics, vol.~38, Cambridge University Press, Cambridge, 1994.
  \MR{1269324 (95f:18001)}

\bibitem[Wei13]{MR3076731}
\bysame, \emph{The {$K$}-book}, Graduate Studies in Mathematics, vol. 145,
  American Mathematical Society, Providence, RI, 2013, An introduction to
  algebraic $K$-theory. \MR{3076731}

\bibitem[Whi50]{MR35997}
J.~H.~C. Whitehead, \emph{A certain exact sequence}, Ann. of Math. (2)
  \textbf{52} (1950), 51--110. \MR{35997}

\end{thebibliography}

\end{document}